\documentclass{article}
\usepackage[utf8]{inputenc}

\usepackage{amsmath}
\usepackage{comment}
\usepackage{leftidx}
\usepackage{multicol}
\usepackage{amssymb}
\usepackage{stmaryrd}
\usepackage{accents}

\usepackage{xfrac}
\usepackage{amsthm}
\usepackage{xcolor}
\usepackage{sectsty}
\usepackage{relsize}
\usepackage{indentfirst}
\usepackage{tikz}
\tikzset{shorten <>/.style={shorten >=#1,shorten <=#1}}

\usetikzlibrary{matrix,arrows,decorations.pathmorphing}
\usepackage{tikz-cd}

\usepackage{quiver}

\newcommand{\cod}
 {{\rm cod}}

\newcommand{\Ind}
 {{\rm Ind}}

\newcommand{\comp}
 {\circ}

\newcommand{\Cont}
 {{\bf Cont}}

\newcommand{\dom}
 {{\rm dom}}

\newcommand{\li}
{{\textup{lim } }}

\newcommand{\lan}
{{\textup{lan } }}

\newcommand{\ran}
{{\textup{ran } }}

\newcommand{\colim}
{{\textup{colim}}}

\newcommand{\comma}[2]
{\mbox{$(#1\!\downarrow\!#2)$}}

\newcommand{\empstg}
 {[\,]}

\newcommand{\epi}
 {\twoheadrightarrow}

\newcommand{\hy}
 {\mbox{-}}

\newcommand{\im}
 {{\rm im}}

\newcommand{\imp}
 {\!\Rightarrow\!}

\newcommand{\mono}
 {\rightarrowtail}

\newcommand{\ob}
 {{\rm ob}}
 
 \newcommand{\Hom}
 {{\rm Hom}}

\newcommand{\op}
 {^{\rm op}}

\newcommand{\Set}
 {{\bf Set }}

\newcommand{\Sh}
 {{\bf Sh}}

\newcommand{\sh}
 {{\bf sh}}

\newcommand{\Sub}
 {{\rm Sub}}
 
\newcommand{\Spec}
 {{\bf Spec}}

\usepackage{cleveref}

\makeatletter
\newcommand{\bigdoublevee}{\big@doubleop{\bigvee}}
\newcommand{\bigdoublewedge}{\big@doubleop{\bigwedge}}
\newcommand{\big@doubleop}[1]{%
  \DOTSB\mathop{\mathpalette\big@doubleop@aux{#1}}\slimits@
}

\newcommand\big@doubleop@aux[2]{%
  \sbox\z@{$\m@th#1#2$}%
  \makebox[1.35\wd\z@][s]{$\m@th#1#2\hss#2$}%
}
\makeatother

\DeclareFontFamily{U}{min}{}
\DeclareFontShape{U}{min}{m}{n}{<-> udmj30}{}

\DeclareUnicodeCharacter{3088}{\yo}

\usepackage{geometry}
\newtheorem{theorem}{Theorem}[section]
\theoremstyle{proposition}
\newtheorem{proposition}[theorem]{Proposition}

\newtheorem{corollary'}[theorem]{Corollary}
\newtheorem{lemma}[theorem]{Lemma}
\theoremstyle{definition}
\newtheorem{definition}[theorem]{Definition}

\theoremstyle{remark}
\newtheorem*{remark}{Remark}

\usepackage[backend=biber,
sorting=nty
]{biblatex}
\nocite{*} 

\sectionfont{\large}
\subsectionfont{\normalsize}

\title{On Diers's theory of Spectrum II \\ Geometries and dualities}
\author{Axel Osmond}
\date{February 2019}

\begin{document}

\maketitle

\begin{abstract}
    This second part comes to the construction of the spectrum associated to a situation of multi-adjunction. Exploiting a geometric understanding of its multi-versal property, the spectrum of an object is obtained as the spaces of local unit equipped with a topology provided by orthogonality aspects. After recalling Diers original construction, this paper introduces new material. First we explain how the situation of multi-adjunction can be corrected in a situation of adjunction between categories of modeled spaces as in the topos-theoretic approach. Then we come to the 2-functorial aspects of the process relatively to a 2-category of Diers contexts. We propose an axiomatization of the notion of spectral duality through morphisms between fibrations over a category of spatial objects, and show how such situations get back right multi-adjoint functors.
\end{abstract}

\section*{Introduction}


This paper is the sequel of \cite{partI} and concerns the explicit construction of spectra from local right multi-adjoint functors as introduced by Diers in \cite{Diers}. While we focused on the algebraic aspects in the first part, synthesizing the main required material about right multi-adjoint and akin notions, this second part will go into the geometric aspects and deploy as explicitly as possible the construction of the spectrum, and how it can be generalized to a more general notion of \emph{spectral dualities}. If the content of the first part was more expository, this second part will introduce new notions and results. \\

Our contexts of interest involve a locally finitely presentable category of ``ambient" objects, for which we want to construct a notion of spectrum, and a category of ``local objects" equiped with a right multi-adjoint functor into the category of ambient objects. Under each object, one has a cone of local units jointly factorizing any morphism from this object to the right multi-adjoint. In general the right multi-adjoint is also far from being full, and hence induces some non trivial orthogonality structure in the ambient category, defining in particular a class of \emph{diagonally universal morphisms} from a condition of left orthogonality relatively to morphisms in the range of the right multi-adjoint. \\

In section 1, we detail Diers original construction of the spectrum for an ambient object. The spectrum will consist of two objects: the \emph{Diers space} constructed at \cref{spectrum} from the set of local units and the poset of diagonally universal morphisms under it ordered for the \emph{factorization order}, and the \emph{structural sheaf} over the Diers space. Local units are to be thought as points of the Diers space, and diagonally universal morphisms, as basic opens of a spectral topology as seen at \cref{spectraltopo}, a point being in a basic open if the corresponding local unit factorizes through the diagonally universal morphism. Then this spectrum is equiped at \cref{structuralsheaf} with a structural sheaf constructed through left Kan extensions from a functor returning the codomain of a diagonally universal morphism of finite presentation. The functor assigning to each object its spectrum will be shown to be left adjoint to a \emph{global section functor} at \cref{Dierstheorem}.\\

In section 2, we generalize the construction of the spectrum to \emph{structured spaces}, which are spaces with a sheaf of objects in the ambient category. The spectrum of a structured space can be constructed by gluing the ordinary Diers spaces of the stalks as in \cref{Dierspacegeneralcase}, with a topology made from the diagonally universal morphisms under value of the sheaf at opens. The structural sheaf, constructed again from Kan extension of a codomain functor, can be shown to behave locally as the structural sheaf of the spectrum of the stalks at \cref{stalkrestriction}. A crucial point is that the local units under a stalk at a point decompose as a colimit of the local units under the values of the sheaf under all its neighborhoods of this point, as seen is \cref{stalk lemma}. Then this generalized spectrum defines a left adjoint to a forgetful functor between categories of structured spaces, as seen at \cref{generalizedadjunction}. The adjunction associated to a right multi-adjoint in \cite{partI}[section 2] appears then as a restriction of this spectral adjunction to the case over discrete spaces. \\

In section 3, we propose an abstraction of the previous construction and examine its 2-functoriality. We introduce at \cref{dualities} an axiomatization of the spectral dualities as morphisms of fibration over the opposite category of spaces, whose codomain is also an opfibration, and endowed with a left adjoint. Then we show at \cref{globalsections} that such a situation always reduces in particular to an adjunction as obtained in section 1. We also define a 2-category of Diers contexts at \cref{diers contexts} and see how they always define such a duality. Finally, we end with a partial converse result at \cref{MRAJ from duality}, showing that a duality as defined above always produce itself a right multi-adjoint at its restriction over the point.

\section{Spectrum construction in the Diers context}

In this first section, we recall and detail as explicitly as possible Diers's construction of the spectra of an object in the context of a right multi-adjoint $ U : \mathcal{A} \rightarrow \mathcal{B}$. In the following, we may refer to objects in $\mathcal{B}$ as \emph{ambient objects}, and to objects and maps of $\mathcal{A}$ as \emph{local objects and maps}.\\

In several papers following older works of Diers, the set $ X_B$ indexing the local units under an object $B$ in the context of a right multi-adjoint has been called ``the spectrum of $B$". If it is true that Diers's approach is point-set contrarily to other ways to construct spectra, it is abusive to reduce the spectrum of an object to the set of local units, since it corresponds more exactly to the set of points of the spectrum. As explained in \cite{Diers}, the spectrum in itself is far more than just a set of points: it is a topological space equipped with a structural sheaf of $\mathcal{B}$-objects with stalks in $\mathcal{A}$, where the diagonally universal morphisms under a given objects correspond to the open of its spectrum. \\

We also hope that, beside all the additional observations we made to complete Diers original construction, for instance about functoriality or the construction of the spectrum of an arbitrary modelled space, the present version will be of use to make easier to access Diers theory for non francophone people, as the seminal work we build on was seemingly never translated into English and remains hard to find in contrast to its importance. \\

First let us present the conditions isolated by Diers enacting the construction of a spectrum from a right multi-adjoint. 
\begin{definition}
A Diers context is the data of a functor $ U : \mathcal{A} \rightarrow \mathcal{B}$ satisfying the following conditions:\begin{itemize}
    \item $ U$ is right multi-adjoint
    \item $ \mathcal{B}$ is locally finitely presentable
    \item any local unit under $ B$ is the filtered colimit of the all the diagonally universal morphisms from $B$ of finite presentation factorizing it. 
\end{itemize} 
\end{definition}

\begin{remark}
The last condition has both a topological and sheaf theoretic interpretation. Topologically, as local units will play the role of the points of the spectrum and diagonally universal morphisms of finite presentation are a basis for the topology, it says that the focal component at a point is the intersection of all its basic open neighborhood. From a sheaf theoretic aspect, it ensures the correct relation between the codomain of a local unit and the stalk of a certain structural sheaf on the spectrum we are going to construct over the spectrum.\\

However, beware that this condition is not automatically fulfilled in a general context. It is fulfilled when the arrows in $\mathcal{A}$ are induced from a right class of a left-generated factorization system. 
\end{remark}

For each object $B$, the slice categories $ B \downarrow \mathcal{B}$ also locally finitely presentable. Now recall from \cite{partI} that $ \mathcal{D}$ denotes the category of diagonally universal morphisms between finitely presented objects in $\mathcal{B}$. It is a small category with finite colimits. For any $B$, consider the category $\mathcal{D}_B$ whose objects are pushouts of morphisms in $\mathcal{D}$ under $B$; again, we saw that this category has finite colimits. \\

The idea of Diers is that the set $X_B$ of local units under a given object $B$, while insufficient in itself to convey the topological and geometric information encoded in a Diers context, can be equiped with a topology defined from $\mathcal{D}_B$. However, in order to get a topological space, we need to forget about the categorical structure of $\mathcal{D}_B $ and see it as a poset. This is the purpose of the following. 

\begin{definition}
The set $\mathcal{D}_B$ can be equiped with the \emph{order of factorization}: for $ n_1 : B \rightarrow C_1$ and $ n_2 : B \rightarrow C_2$, define $ n_1  \leq n_2 $ if and only if $n_2$ factorizes through $n_1$ for some $n$ \[\begin{tikzcd}
B \arrow[]{r}{n_2} \arrow[]{d}[swap]{ n_1} & C_2 \\ C_1 \arrow[dashed]{ru}[swap]{\exists n}
\end{tikzcd} \]
This order extends to arbitrary diagonally universal morphisms under $B$. \\

In particular, this order extends to the set $X_B$ of local units under $B$: that is, for $ x_1 : B \rightarrow U(A_1) $ and $ x_2 : B \rightarrow U(A_2)$ if there exists some $ n : U(A_1) \rightarrow U(A_2)$ such that
\[  \begin{tikzcd}
B \arrow[]{r}{x_2} \arrow[]{d}[swap]{ x_1} & U(A_2) \\ U(A_1) \arrow[dashed]{ru}[swap]{\exists n}
\end{tikzcd} \]
\end{definition}
\begin{remark}
By left cancellation of diagonally universal morphism, any map $ n$ as in the diagram above must be diagonally universal. Moreover, $\mathcal{D}_B$ inherits posetal analogs of its categorical structure. In particular, $\mathcal{D}_B$ has finite joins: for $ n_1 : B \rightarrow C_1$ and $ n_2 : B \rightarrow C_2$, let $n_1 \vee n_2$ denote the map induced under $B$ by the pushout 
    \[ 
\begin{tikzcd}
B \arrow[d, "n_1"'] \arrow[r, "n_2"] \arrow[rd, "\ulcorner", phantom, very near end] & C_2 \arrow[d] \\
C_1 \arrow[r]                                                         & C_1 +_B C_2  
\end{tikzcd}\]
Then $ n_1 \vee n_2$ is really the join in the poset $ \mathcal{D}_B$ as for any $n : B \rightarrow C$, if $ n_1, \, n_2 \leq C$, then there are $ m_1, \,m_2$ such that $m_1n_1 = n =m_2 n_2$, so that $m_1$ and $m_2$ produce a commutative square, hence a map $ \langle m_1, m_2 \rangle : C_1 +_B C_2 \rightarrow n$ attesting that $ n_1 \vee n_2 \leq n$, and conversely. In particular, the opposite poset $ \mathcal{D}_B^{op}$ is a $\wedge$-semilattice with top element. 
\end{remark}

\begin{definition}
For each $ n \in \mathcal{D}_B$, define the set 
\[ X_n = \{ x \in X_B \mid \, n \leq x \} \]
This definition extends naturally for arbitrary diagonally universal morphisms. Conversely, for each $x \in X_B$, define the set 
\[ V_x = \{ n \in \mathcal{D}_B \mid \, n \leq x \} \]
\end{definition}

\begin{remark}
Then in a Diers context, any candidate $x$ decomposes as  colimit in the category $ B \downarrow \mathcal{B}:$
\[ B \stackrel{x}{\rightarrow} U(A_x) = \underset{n \in V_x}{\colim} \; B \stackrel{n}{\rightarrow} C   \]
Then from expression of filtered colimits in the coslice, we have that 
\[ U(A_x) \simeq \underset{n \in V_x}{\colim} \; \cod(n) \]
\end{remark}

We also list here some obvious, yet meaningful properties of $D$:

\begin{proposition}
We have the following, for any object $B$ of $\mathcal{B}$:\begin{itemize}
    \item If $ n_1 \leq n_2$ in $\mathcal{D}_B$, then $D_{n_2} \subseteq D_{n_1}$
    \item If $ x_1 \leq x_2$ in $X_B$, then $ V_{x_1} \subseteq V_{x_2}$
    \item $ D(1_B) = X_B$
    \item for $ n_1, \, n_2 \in \mathcal{D}_n$, we have $ D(n_1) \cap D(n_2) = D(n_1 \vee n_2)$ 
\end{itemize}
\end{proposition}

\begin{remark}
The fourth item says in particular that $ \{ D_n \mid \, n \in \mathcal{D}_B \} $ is a basis for a topology as it is closed under intersection. 
\end{remark}

\begin{definition}\label{spectraltopo}
The \emph{spectral topology} of $B$ is the topology on the set $ X_B$ generated as
\[ \tau_B = \langle \{ D_n \mid \, n \in \mathcal{D}_n \}\rangle  \]
In particular we have a monotone, $\wedge$-preserving map $ D : \mathcal{D}_B^{op} \rightarrow \tau_B$.
\end{definition}

\begin{remark}
One could also process in a point-free way as follows: from the poset $ \mathcal{D}_B$, generated the free frame $ \mathcal{F}_B = \langle \mathcal{D}_B^{op} \rangle_{\mathcal{F}rm}$, then equip it with the cover $J_B$ defined as  
\[ J_B(u) = \{ (n_i)_{i \in I} \mid \, \bigcup_{i \in I} D_{n_i} = u \} \]
Beware however that, without some specific hypothesis, the spectrum $ (X_B, \tau_B)$ may not be $T_0$-separated and hence not sober, so that $X_B$ may not coincide with $ pt(\mathcal{D}_B, J_B)$. Conversely, the frame $ \mathcal{F}_B$ may not be spatial if the map $ D$ fails to be injective, or to be order reflecting.
\end{remark}

\begin{remark}
Beware that, for an arbitrary Diers context, the category of diagonally universal morphisms may not be locally finitely presentable, and the orthogonality structure generated as in \cite{partI} may not be left generated. In this case some diagonally universal morphisms may not be obtained as filtered colimit of finitely presented one: then the inclusion
\[ \Ind(^\perp U(\overrightarrow{\mathcal{A}}) \cap \overrightarrow{\mathcal{B}_{fp}}) \subseteq \, ^\perp\! U(\overrightarrow{\mathcal{A}})  \]
is strict, and the factorization system induced by the small object argument as in \cite{Anel} returns a wider class on the right 
\[ (^\perp U(\overrightarrow{\mathcal{A}}))^\perp \subseteq (^\perp U(\overrightarrow{\mathcal{A}}) \cap \overrightarrow{\mathcal{B}_{fp}})^\perp \]
So in the general case, we will have at some point to distinguish between arbitrary diagonally universal morphisms and the one that are obtained as filtered colimits of basic ones. We call morphisms in $ \Ind(^\perp U(\overrightarrow{\mathcal{A}}) \cap \overrightarrow{\mathcal{B}_{fp}})$ \emph{ axiomatisable diagonally universal morphisms}, because we saw in \cite{partI}[section 3] they are models of a finite-limit theory. We also consider $ \Ind(\mathcal{D}_B)$, {which are the axiomatisable diagonally universal morphisms under $B$.}\\

We recall that we said that a Diers context $U$ to \emph{diagonally axiomatisable} when one, hence both, of the two inclusion above are equalities. Then in this case, the category $\mathcal{D}iag =\Ind(^\perp U(\overrightarrow{\mathcal{A}}) \cap \overrightarrow{\mathcal{B}_{fp}}) $ of diagonally universal morphisms is locally finitely presented, and for each $B$ in $\mathcal{B}$, so is the category $ \mathcal{D}iag_B$ of diagonally universal morphisms under $B$, in this case we do have $ \mathcal{D}iag_B= \Ind(\mathcal{D}_B)$. 
\end{remark}

Now observe that for each $B$, we can extend $ D$ to axiomatisable diagonally universal morphisms by computing the pointwise left Kan extension of the functor $D$
\[ 
\begin{tikzcd}
\mathcal{D}_B \arrow[r, "D"] \arrow[d, "\iota"', hook] & \tau_B^{op} \\
\Ind(\mathcal{D}_B) \arrow[ru, "{\lan_\iota \, D}"', " \simeq"{inner sep=5pt}]   &            
\end{tikzcd}\]
where the canonical 2-cell is invertible by full faithfulness of the dense inclusion. This extension expresses as a codirected intersection
\[ \lan_\iota \, D (l) = \bigcap_{ \mathcal{D}_B\downarrow l} D_n \]
ensuring that the canonical cocone of $l$ in $\Ind(\mathcal{D}_B)$ is sent to a canonical intersection of basic neighborhoods. In the following, we also abusively denote $\lan_\iota \, D (l)$ as $D_l$. The fact that the intersection range over the canonical cone do not modify the intersection as intersections are idempotent, so we could also rewrite this identity as 
\[ D_l = \bigcap_{n \leq l \atop n \in \mathcal{D}_B} D_n  \]
that is, by considering only the order on diagonally universal morphisms rather than their categorical structure. \\

\begin{remark}
 The equality above can be understood as follows: suppose that $x \in \bigcap_{m :n \rightarrow l \atop n \in \mathcal{D}_B} D_n  $, that is, that for each $m : n \leq l$ we have a $ q_m : n \rightarrow x$; then as $ l = \colim_{m \in \mathcal{D}_B\downarrow l} \dom(m)$, we have a canonical morphism $ \langle q_m \rangle_{ \mathcal{D}_B\downarrow l} : l \rightarrow x$ ensuring that $l \leq x$, that is, $x \in D_l$. One has to compose any $l \rightarrow x$ with the canonical cone of $f$ to see the converse inclusion.
\end{remark}

\begin{remark}
The induced specialization order on $X_B$ is exactly the restriction of the order of factorization $ \leq $ on the local units because if $x_1 \leq x_2 $ in $X_B$, then for any $n \in \mathcal{D}_n$, $ x_1 \in D_n$ implies $ x_2 \in D_n$, that is, $V_{x_1} \subseteq V_{x_2}$, and as the $D_n $ are a basis for the topology, so the same is true for any open. Remark we can also restrict $D$ at any $V_x$. 
\end{remark}

\begin{remark}
Observe that one could extends the functor $ D$ to arbitrary diagonally universal morphisms under $B$. In the case where they are left generated, then they are filtered colimits of diagonally universal morphisms of $B$.

\end{remark}

\begin{definition}\label{spectrum}
For a Diers context $U : \mathcal{A} \rightarrow \mathcal{B}$, any object $B$ of $\mathcal{B}$, define the \emph{Spectrum of $B$} as the space $ \Spec(B)= (X_B, \tau_B)$. We call also this space the \emph{Diers space of $B$}. 
\end{definition}

Now let us explicit the functorial aspects of this construction. Let $ f : B_1 \rightarrow B_2$ in $\mathcal{B}$. Then for any point $ x : B_2  \rightarrow U(A)$ the factorization of the composite

\[\begin{tikzcd}
B_1 \arrow[]{r}{f} \arrow[]{rd}[swap]{\eta^A_{xf}} & B_2 \arrow[]{r}{x} & U(A) \\ & UL_A(xf) \arrow[]{ru}[swap]{U(L_A(xf)} &
\end{tikzcd}\]
defines uniquely a point $ \eta_{xf} : B_1 \rightarrow UL_A(nf)$ of $X_{B_1}$ and we have to pose
\[\begin{array}{rrcl}
    \Spec(f) : & X_{B_2} & \rightarrow  & X_{B_1} \\
     &  B_2  \stackrel{x}{\rightarrow} U(A) & \longmapsto &  B_1 \stackrel{\eta_{xf}}{\rightarrow} U(L_A(xf))
\end{array}\]

\begin{proposition}
For any $ f : B_1 \rightarrow B_2 $ in $\mathcal{B}$ and any $n$ in $\mathcal{D}_{B_1}$, we have $ \Spec(f)^{-1}(D^{B_1}_n) = D^{B_2}_{f_*n}$, that is, $\Spec(f)$ restricts to the basis of the topology and we have 
\[ 
\begin{tikzcd}[column sep=large]
\mathcal{D}_{B_1}^{op} \arrow[r, "(f_*)^{op}"] \arrow[d, "D^{B_1}"']       & \mathcal{D}_{B_2}^{op} \arrow[d, "D^{B_2}"] \\
\tau_{B_1} \arrow[r, "\Spec(f)^{-1}"'] & \tau_{B_2}             
\end{tikzcd} \]
In particular, the map $ \Spec(f) : \Spec(B_2) \rightarrow \Spec(B_1)$ is continuous.
\end{proposition}

\begin{proof}
Now for a diagonally universal morphism of finite presentation $ n \in \mathcal{D}_{B_1} $ one has $ n \leq \eta_{xf} $ if and only if $ f_* n \leq x$. That is, $\Spec(f)(x) \in D^{B_1}_n$ if $ x \in D^{B_2}_{f_*n}$. This proves that $ \Spec(f) $ restricts to the basis as stated above. But testing that inverse image of basic opens are open is sufficient for continuity.
\end{proof}


\begin{definition}
The spaces of the form $ (X_B, \tau_B)$ for $B$ in $ \mathcal{B}$ are collectively called \emph{$\emph{U}$-spectral spaces} or \emph{ the Diers spaces of U}, while the maps of the form $\Spec(f)$ are called \emph{$U$-spectral maps}.  
\end{definition}

The power of Diers approach resides also in the way one can characterize topological properties of the Diers spaces from purely functorial properties of $U$. We list here the following characterizations from \cite{Diers}[Section 7].

\begin{proposition}
Let $U : \mathcal{A} \rightarrow \mathcal{B}$ be a Diers context. Then: \begin{enumerate}
    \item Diers spaces are $T_0$ iff $U$ is conservative
    \item Diers spaces are $T_1$ iff $U$ is full and faithful
    \item Diers spaces are $T_2$ iff they are zero-dimensional iff $U$ is full and faithful and moreover any object $B$ such that $ X_B = \emptyset$ is a finitely presented object in $\mathcal{B}$
    \item Diers spaces are compact iff $U$ lifts ultraproducts
    \item Diers spaces are boolean iff $U$ is full and faithful and lifts ultraproducts
\end{enumerate}
\end{proposition}

\begin{remark}
In several examples of this construction, considering only Diers spaces together with spectral maps between them is sufficient to give rise to a duality, as long as one can characterize their topological structure and how it is preserved by spectral maps. This process is suited in particular for situations where diagonally universal morphisms are quotients in correspondence with some specific kind of congruences or ideals, forming posets in a certain variety, so that the corresponding Diers spaces can be characterized through algebraic properties of the basis of their spectral topology. This is more in the spirit of concrete dualities, though such concrete dualities always seem to correspond to a Diers context. However in this paper, following Diers approach and more generally a vision of spectra more related to algebraic geometry than concrete dualities, the kind of dualities we are going to obtain makes use of geometric information attached to the Diers space to reconstruct algebras, rather than extracting them from information in their topology, as it is done in Grothendieck duality for commutative ring. In particular, the spectrum of an object will not just correspond to a Diers space, but needs also to bear a structural sheaf to achieve the universal property expected from the spectrum. 
\end{remark} 

We have constructed for each object $B $ in $\mathcal{B}$ a space from the data associated to the right multi-adjoint $U$. This space was defined as having as set of points the set indexing its local units under $B$, and as basis for the topology the opposite of the underlying poset of the category of diagonally universal morphisms of finite presentation under $B$. But observe that this construction does not retain enough informations in some sense, as it forgets about the objects it was defined from: in particular, while we record the set indexing the local units (and the specialization order induced by the topology), the precise value of the codomains of the local units and the local units themselves are not remembered. So we need to attach those data to the Diers space of $B$. This is the purpose of the \emph{structural sheaf of $B$}.\\ 

This requires some preliminary remarks on sheaves of objects in a category other than $\Set$, topics for which a standard reference is \cite{KashiwaraSchapira}[section 17.3]. Recall first that a \emph{presheaf} of $\mathcal{B}$-objects on a topological space $(X,\tau)$ is a contravariant functor $\mathbb{B}  :\tau^{op} \rightarrow \mathcal{B}$, where the induced maps $\mathbb{B}_{u_2} \rightarrow \mathbb{B}_{u_1}$ for $ u_1 \subseteq u_2$ are called the \emph{projections} and are denoted $ \rho^{u_1}_{u_2}$; for a point $ x \in X$, the \emph{stalk} at $x$ is the filtered colimit in $\mathcal{B}$ of the values on neighborhoods of $x$
\[ \mathbb{B}\!\mid_x = \underset{x \in u}{\colim} \, \mathbb{B}(u) \]
A \emph{sheaf} of $\mathcal{B}$-objects on $(X,\tau)$ is a presheaf with the following \emph{descent condition}: for any $ u \in \tau$ and $(u_i)_{i \in I}$ such that $ u= \bigcup_{i \in I} u_i$ the projections $ \rho^{u}_{u_i}$ exhibit $\mathbb{B}(u)$ as a limit in $\mathcal{B}$
\[ \mathbb{B}(u) = \lim \Big{ (} \prod_{i \in I} \mathbb{B}(u_i) \rightrightarrows \prod_{i \in I} \mathbb{B}(u_i \cap u_j)   \Big{ )} \]
where the two maps are induced from the restrictions $ \rho^{u_i}_{u_i \cap u_j}$ and $\rho^{u_j}_{u_i \cap u_j}$ respectively. Observe that those definitions require the completeness and cocompleteness of $\mathcal{B}$, which are satisfied in the case where $\mathcal{B}$ is locally finitely presentable as in our hypothesis. While one can define sheaves in categories more general than locally finitely presentable categories, we restrict to the later in our case. \\

Hence for any topological space $(X,\tau)$ we can consider the category $\Sh_\mathcal{B}(X,\tau)$ of \emph{sheaves of $\mathcal{B}$-objects} over $ (X,\tau)$, where morphisms are natural transformations. This is a full subcategory of the category $ [ \tau^{op}, \mathcal{B}]$. Moreover we can define a sheafification functor left adjoint to this inclusion 
\[\begin{tikzcd}
	{\Sh_\mathcal{B}(X,\tau)} && {[\tau^{op}, \mathcal{B}]} && {}
	\arrow[""{name=0, inner sep=0}, from=1-1, to=1-3, curve={height=20pt}, hook]
	\arrow["{\mathfrak{a}_\tau}"{name=1, swap}, from=1-3, to=1-1, curve={height=20pt}]
	\arrow["\dashv"{rotate=-90}, from=1, to=0, phantom]
\end{tikzcd}\]
where we denote  \[ \mathbb{B} \stackrel{\gamma}{\rightarrow} \mathfrak{a}_\tau \mathbb{B} \] the unit of the sheafification. This sheafification functor preserves finite limits and small colimits; it also does not modify stalks, that is, for any presheaf $\mathbb{B}$ and $ x \in X$, $ (\mathfrak{a}_\tau \mathbb{B})\!\mid_x = \mathbb{B}\!\mid_x$. Moreover, recall that limits in category of sheaves are computed pointwisely, as well as filtered colimits - though computation of arbitrary colimits may not be pointwise. \\

Now for a continuous map $ f : (X_1, \tau_1) \rightarrow (X_2, \tau_2)$, as well as inverse and direct images exist for $\Set$ valued sheaves, we can define inverse and direct image for sheaves of $ \mathcal{B}$-objects
\[ 
\begin{tikzcd}
{\Sh_{\mathcal{B}}(X_1,\tau_1)} \arrow[rr, "f_*"', bend right=20] & \perp & {\Sh_{\mathcal{B}}(X_2,\tau_2)} \arrow[ll, "f^*"', bend right=20]
\end{tikzcd} \]
Here $ f_*$ is precomposition with $ f^{-1}: \tau_2 \rightarrow \tau_1 $ so that in each $ u \in \tau_2$, and $ \mathbb{B}$ in $\Sh_{\mathcal{B}}(X_1, \tau_1)$ we have $ f_*\mathbb{B}(u) = \mathbb{B}(f^{-1}(u))$; on the other hand, for each $ \mathbb{B}$ in $\Sh_{\mathcal{B}}(X_2, \tau_2)$, the inverse image $ f^*\mathbb{B} $ is obtained from the left Kan extension   
\[\begin{tikzcd}[column sep=large]
	{\tau_2^{op}} & {\mathcal{B}} \\
	{\tau_1^{op}}
	\arrow["{f^{-1}}"', from=1-1, to=2-1]
	\arrow["{\mathbb{B}}", ""{name=0, inner sep=1pt, below}, from=1-1, to=1-2]
	\arrow["{\lan_{f^{-1}}\mathbb{B}}"', ""{name=1}, from=2-1, to=1-2]
	\arrow[Rightarrow, "{\theta}"'{near end}, from=0, to=1, shorten >=-5pt, curve={height=2pt}]
\end{tikzcd}\]
whose value at $ u \in \tau_1$ is given as the colimit $ f^*\mathbb{B}(u) = \colim_{ u \subseteq f^{-1}(v)} \mathbb{B}(v)$,
followed by sheafification, that is
\[ f^*\mathbb{B}(u) = \mathfrak{a}_{\tau_1}\, \lan_{f^-1}\mathbb{B} \]
Its stalk at $x \in X_1$ is $ \mathbb{B}\!\mid_{f(x)}$.\\

Now we turn to the construction of the structural sheaf on the Diers space $(X_B, \tau_B)$ for an object $B$ in $\mathcal{B}$.

\begin{definition}
For a given $ B \in \mathcal{B}$ the \emph{structural presheaf} is defined as the left Kan extension of the codomain functor along the functor $ D$ : 
\[ \begin{tikzcd}[sep=large]
\mathcal{D}_B \arrow[r, ""{name=A, inner sep=0.1pt, below}, "\cod"] \arrow[d, "D"'] & \mathcal{B} \\ \tau_B^{op} \arrow[ru, ""{name=B, inner sep=0.1pt}, "\overline{B} = \lan_D \cod"'] \arrow[from=A, to=B, Rightarrow, bend right=20, "\zeta"']
\end{tikzcd} \]
That is for any $ u \in \tau_B$ : 
\[  \overline{B}(u) = \underset{ u \subseteq D_n }{\colim}\; \cod(n) \]
and it is equiped with its universal natural transformation $ \zeta : \cod \Rightarrow \overline{B}D$ defined as the collection of maps
\[ \zeta = (\zeta_n : \cod(n) \rightarrow \underset{ D_n \subseteq D_m }{\colim} \cod(m))_{n \in  \mathcal{D}_B }\]
defined as a subcocone of the pointwise colimit expression of the left Kan extension ranging over the elements $D_n = D_n$ amongst all the inequalities $ D_m \supseteq D_n$. 
\end{definition}

\begin{remark}
Beware that in general the natural transformation $\zeta$ may not be a pointwise isomorphism: this will be the case if and only if $ D$ is order-reflecting as a poset map. Then in this case, we can also define a natural transformation in the opposite direction from the fact that $ D_n \subseteq D_m$ iff $n $ factorizes through $ m $, inducing a map from the property of colimits
\[ \begin{tikzcd}
B \arrow[]{r}{n} \arrow[]{rd}[swap]{m} & \cod(n) & \arrow[dashed]{l}[swap]{\exists \xi_n} \underset{m \leq n}{\colim} \, \cod(m) \\ & \cod(m) \arrow[]{u}{} \arrow[]{ru}{} & 
\end{tikzcd} \]
which is natural in $ n$ by universal property of colimits, and this provides a natural inverse to $\zeta$.
\end{remark}

\begin{definition}\label{structuralsheaf}
The structural sheaf for $B$ is the sheafification $\overline{B} \stackrel{\gamma }{\rightarrow}\widetilde{B} = \mathfrak{a}_\tau \overline{B}$, and we have for each $ n : B \rightarrow C$ a morphism in $\mathbb{B}$
\[ C \stackrel{\zeta_n}{\longrightarrow} \overline{B}(D_n) \stackrel{\gamma_n}{\longrightarrow} \widetilde{B}(D_n) \]
and, in particular, a universal map
\[  B \stackrel{\zeta_{1_B}}{\rightarrow} \overline{B}(1_B) \stackrel{\gamma_{1_B}}{\rightarrow} \widetilde{B}(1_B) =\Gamma\widetilde{B} \]
we denote as $ \eta_B : B \rightarrow \Gamma \widetilde{B}$
\end{definition}

Now we come to the local behavior of this structural sheaf by defining separately the restriction of the structural presheaf to the neighborhood. Recall that the upset of $x$ for the specialization order is the saturated compact obtained as

\[ \uparrow_\sqsubseteq x= \bigcap \mathcal{V}_x \]

This subset of $X_B$ is called the \emph{focal component} of $X_B$ at $x$. The restriction codomain of arrows to $V_x$ comes equiped with a colimiting cone with submit $U(A_x)$\[ \phi :  \cod\mid_{V_x} \Rightarrow \Delta_{U(A_x)} \] from the colimit decomposition of Diers condition exhibiting $ U(A_x) = \colim_{n \in V_x} \cod(n)$. Then observe that $D\mid_x : V_x \hookrightarrow \mathcal{V}^{op}_x$ is cofinal as any neighborhood of $x$ contains a basic neighborhood of the form $D_n$ with $n\leq x$.\\

\begin{proposition}
For any $B $ in $\mathcal{B}$ and $ x \in X_B$, $ \widetilde{B}\!\mid_x = U(A_x)$
\end{proposition}

\begin{proof}
This comes from the expression of Kan extension as colimit and the condition in the Dier context that candidate are obtained as filtered colimits of morphism of finite presentation. First recall that sheafification does not modify the stalks, so that $ \widetilde{B}\!\mid_x = \overline{B}\!\mid_x$. Now we have 
\[ \overline{B}\!\mid_x = \underset{n \leq x}{\colim} \, \overline{B}(D_n) = \underset{n \leq x}{\colim} \, \underset{D_n \subseteq D_m}{\colim} \cod(m) \]
But the diagram made of all the $ n \leq x$ is cofinal in the indexing diagram of this colimit, and then it coincide with the colimit $ \colim_{n \leq x} \cod(n)$ which is $U(A_x)$ by Diers condition. 
\end{proof}

Now, for a morphism $ f: B_1 \rightarrow B_2$ in $\mathcal{B}$, we construct a morphism of sheaf between the structural sheaves associated to $B_1$ and $B_2$. For any $n$ in $\mathcal{D}_{B_1}$ we have a map $ \gamma_{D_n} \zeta_{n}$ producing a composite map $ \sigma_n$ as below
\[\begin{tikzcd}
	{B_1} & {B_2} \\
	{C} & {f_*C} & {\widetilde{B}_2(f^{-1}(D_n))}
	\arrow["{n}"', from=1-1, to=2-1]
	\arrow["{f}", from=1-1, to=1-2]
	\arrow["{\rho^{X_{B_2}}_{f^{-1}(D_n)} \eta_{B_2}}", from=1-2, to=2-3]
	\arrow["{f_*n}" description, from=1-2, to=2-2]
	\arrow[from=2-1, to=2-2]
	\arrow["{\gamma_{D_n}\zeta_n}", from=2-2, to=2-3]
	\arrow["{\sigma_n}"', from=2-1, to=2-3, curve={height=18pt}]
\end{tikzcd}\] 
and the data of all the $(\sigma_n : \cod(n) \rightarrow \widetilde{B}_2(X_f^{-1}(D_n))_{n \in \mathcal{D}_{B_1}}$ can be shown to provide a natural transformation $ \cod \Rightarrow {X_f}_*\widetilde{B}_2 $, so the left Kan extension produces a factorization
\[\begin{tikzcd}[column sep=large]
	{\mathcal{D}_{B_1}} & {} & {\mathcal{B}} \\
	{\tau_{B_1}^{op}} & {} \\
	{\tau_{B_2}^{op}}
	\arrow["{D}"', from=1-1, to=2-1]
	\arrow["{X_f^{-1}}"', from=2-1, to=3-1]
	\arrow["{\cod}"{name=0}, from=1-1, to=1-3]
	\arrow["{\widetilde{B}_2}"{name=1, swap}, from=3-1, to=1-3, curve={height=18pt}]
	\arrow["{\overline{B}_1}"{name=2, description}, from=2-1, to=1-3]
	\arrow[Rightarrow, "{\sigma}"', from=0, to=1, shift right=9, curve={height=6pt}, shorten <=3pt, shorten >=3pt, crossing over]
	\arrow[Rightarrow, "{\zeta}"{name=3, swap}, from=0, to=2, shorten <=2pt, shorten >=2pt]
	\arrow[Rightarrow, "{\overline{f}}", from=2, to=1, shorten <=3pt, shorten >=3pt, dashed]
\end{tikzcd}\]
and as the direct image $ {X_f}_*\widetilde{B}_2$ is still a sheaf, this induced map $ \overline{f}$ itself factorizes through the sheafification as
\[\begin{tikzcd}
	{\overline{B}_1} & {{X_f}_*\widetilde{B}_2} \\
	{\widetilde{B}_1}
	\arrow["{\overline{f}}", from=1-1, to=1-2]
	\arrow["{\gamma}"', from=1-1, to=2-1]
	\arrow["{\widetilde{f}^\sharp}"', from=2-1, to=1-2]
\end{tikzcd}\]
This morphism $ \widetilde{f}^\sharp $ now corresponds itself to a morphism of sheaf $ \widetilde{f}^\flat : {X_f}^*\widetilde{B}_1 \rightarrow \widetilde{B}_2$. \\ 

To turn this construction into a functor, we need first to determine where it would land. We constructed a space equiped with a distinguished structural sheaf of $\mathcal{B}$-objects with stalks in the range of $U$. 
\begin{definition}
Define the category $U-Spaces$ of \emph{$ U$-spaces} as the category whose\begin{itemize}
    \item objects consist of triples $ ((X,\tau), \mathbb{A}, (A_x)_{x \in X})$ where $ (X,\tau)$ is a topological space, $ \mathbb{A}$ is a sheaf of $\mathcal{B}$-objects over $(X,\tau)$, and $ (A_x)_{x \in X}$ is a family of objects in $\mathcal{A}$ such that the stalks of $\mathbb{A}$ satisfy $ \mathbb{A}\!\mid_x \simeq U(A_x)$ for any $x \in X$ 
    \item morphisms $((X_1,\tau_1), \mathbb{A}_1, (A^1_x)_{x \in X_1}) \rightarrow ((X_2,\tau_2), \mathbb{A}_2, (A^2_x)_{x \in X_2})$ consist of triples $(f, \phi, (u_x)_{x \in X_1})$ where $ f :(X_2,\tau_2) \rightarrow (X_1,\tau_1)$ is continuous, $\phi$ consists of a pair of morphisms of sheaves 
    \[ ( \phi^\flat : f^*\mathbb{A}_1 \rightarrow \mathbb{A}_2, \, \phi^\sharp : \mathbb{A}_1 \rightarrow f_*\mathbb{A}_2) \]
    corresponding through the adjunction $ f^* \dashv f_*$ and $ u_x : A^1_{f(x)} \rightarrow A_x^2$ is an arrow in $\mathcal{A}$ such we have at stalks $\phi^\flat_x = U(u_x) $ for each $ x \in X_1$. The morphisms $ (\phi^\flat, \phi^\sharp)$ are called the \emph{inverse and direct comorphism part} of the morphism of $U$-spaces. 
\end{itemize}
\end{definition}

\begin{remark}
In this definition, $U$-spaces are not just spaces with sheaves that have their stalk in the essential image of $U$: in fact we attach to them a specification of which objects of $\mathcal{A}$ their stalks come from. Similarly for morphisms where we impose that their inverse image part comes from an arrow in $\mathcal{A}$.
\end{remark}

\begin{remark}
For any morphism of $U$-spaces $ (f, \phi)$, the inverse image $ f^*\mathbb{A}$ still has its stalks in $ \mathcal{A}$ as for any point in $X_2$, we have $ f^*\mathbb{A}_1\!\mid_x = \mathbb{A}_1\!\mid_{f(x)}$. Moreover it is a standard result that inverse image preserves finite limits. However we cannot control the stalks of the direct image $ f_*\mathbb{A}_2$, which may not be in the range of $U$. This is related to the fact that, while sheaves $\mathcal{B}$-objects, as object of a locally presentable categories, are stable under inverse and direct image, objects of $\mathcal{A}$ may be in a more wild class of objects, and in general sheaves of objects in non locally finitely presentable categories are not anymore stable under direct image. This is typically true when objects in $\mathcal{A}$ are model of a geometric sketch with non trivial inductive part that cannot be preserved by direct image for they are right adjoints and hence need not preserve colimits. 
\end{remark}

\begin{remark}
In a morphism of $U$-spaces $((X_1,\tau_1), \mathbb{A}_1, (A^1_x)_{x \in X_1}) \rightarrow ((X_2,\tau_2), \mathbb{A}_2, (A^2_x)_{x \in X_2})$, the inverse and direct image parts are related as follows. For any $x \in X_2 $ and $ u \in \tau_1 $ such that $ f(x) \in u$, that is, $x \in f^{-1}(u)$, then we have $ f_*\mathbb{A}_2(u) = \mathbb{A}(f^{-1}(u))$, and we have the composite map  
\[\begin{tikzcd}
	{\mathbb{A}_1(u)} & {\mathbb{A}_2(f^{-1}(u))} \\
	& {\mathbb{A}_2\!\mid_x}
	\arrow["{\phi^\sharp_u}", from=1-1, to=1-2]
	\arrow["{\rho^{f^{-1}(u)}_x}", from=1-2, to=2-2]
	\arrow[from=1-1, to=2-2]
\end{tikzcd}\]
But now the stalk of $\mathbb{A}_1$ at $ f(x)$ is the filtered colimit $ \mathbb{A}_1\! \mid_{f(x)} = \colim_{f(x) \in u} \mathbb{A}_1(u)$, so the cone made of all the maps above factorizes uniquely through this colimits, and this produces the desired value of the inverse image part of $\phi$ at $x$, so that we have a commutation
\[\begin{tikzcd}
	{\mathbb{A}_1(u)} & {\mathbb{A}_2(f^{-1}(u))} \\
	{\mathbb{A}_1\!\mid_{f(x)}} & {\mathbb{A}_2\!\mid_x}
	\arrow["{\phi^\sharp_u}", from=1-1, to=1-2]
	\arrow["{\rho^{f^{-1}(u)}_x}", from=1-2, to=2-2]
	\arrow["{\rho^{u}_{f(x)}}"', from=1-1, to=2-1]
	\arrow["{\phi^\flat_x}"', from=2-1, to=2-2]
\end{tikzcd}\]
where the inverse image part is exhibited as the universal map \[ \phi^\flat_x = \langle \rho^{f^{-1}(u)}_x \phi^\sharp_u \rangle_{f(x) \in u} \]
\end{remark}

To sum up, in the first part of this section, we defined for each object $B $ in $\mathcal{B}$ a $U$-space $ ((X_B, \tau_B), \widetilde{B}, (A_x)_{x \in X_B})$, and for each $ f : B_1 \rightarrow B_2$, a morphism of $U$ space $ (\Spec(f), \widetilde{f})$. 
\begin{definition}
The construction above defines a functor called \emph{the spectrum of U}
\[ \mathcal{B} \stackrel{\Spec}{\longrightarrow } U-Spaces \]
\end{definition}
Now we look at a functor going in the converse direction. Let be a $U$-space $ ((X,\tau), \mathbb{A},(A_x)_{x \in X})$. As $\mathcal{B}$ is locally finitely presentable, the category of sheaves over $ (X,\tau)$ with value in $\mathcal{B}$ is equiped with a global section functor $ \Gamma :  \Sh_\mathbb{B}(X,\tau) \rightarrow \mathcal{B}$ sending a sheaf $ \mathbb{A}$ to the $\mathcal{B}$-object of global sections $ \Gamma \mathbb{A} = \mathbb{A}(X)$. For any $ u \in \tau,$ we denote as $ q^X_u : \Gamma\mathbb{A} \rightarrow \mathbb{A}(u) $, and for any $x \in X$, the stalk at $x$ is obtained as a filtered colimit in $\mathcal{B}$
\[ \mathbb{A}\!\mid_x = \underset{x \in u}{\colim}\, \mathbb{A}(u) \]
and the induced map $ q^X_x : \Gamma \mathbb{A} \rightarrow \mathbb{A}\mid_x $ is the filtered colimit of the maps $ q_u^X$ in the slice $ \Gamma \mathbb{A} \downarrow \mathcal{B}$. \\

Now for a morphism of $U$-spaces $ (f,\phi, (u_x)_{x \in X}): (X_1, \tau_1), \mathbb{A}_1, (A^1_x)_{x \in X_1}) \rightarrow (X_2, \tau_2), \mathbb{A}_2, (A^2_x)_{x \in X_2})$ the direct image part of the morphism of sheaf $\phi^\sharp $ takes at $X_1$ the value $ 
\phi^\sharp_{X_1} : \mathbb{A}_1(X_1) \rightarrow \mathbb{A}_2(f^{-1}(X_1))$, but as $ f^{-1}(X_1) = X_2$, this defines a morphism in $\mathcal{B}$ 
\[ \Gamma \mathbb{A}_1 \stackrel{\Gamma \phi}{\longrightarrow} \Gamma \mathbb{A}_2 \]

\begin{definition}
Those data allow us to define a functor 
\[ U-Spaces \stackrel{\Gamma}{\longrightarrow} \mathcal{B}^{op} \]
sending a $U$-space $((X,\tau), \mathbb{A},(A_x)_{x \in X}) $ on the $\mathcal{B}$-object of global sections $\Gamma\mathbb{A}$ and a morphism $ (f,\phi, (u_x)_{x \in X})$ to $\Gamma \phi$. 
\end{definition}

\begin{theorem}[Diers]\label{Dierstheorem}
Let $U : \mathcal{A} \rightarrow \mathcal{B}$ define a Diers context. Then there is an adjunction
\[ \begin{tikzcd}[column sep=large]
\mathcal{B} \quad \arrow[bend left=20]{rr}{\Spec} & \perp & U-Spaces \arrow[bend left=20]{ll}{\Gamma}
\end{tikzcd} \]

\end{theorem}

\begin{proof}This proof follows \cite{Diers}[3.6.1]. 
Let be $B$ in $\mathcal{B}$, $((X,\tau), \mathbb{A},(A_x)_{x \in X}) $ a $U$-space and $ \phi : B \rightarrow \Gamma \mathbb{A}$. Then for any $ x \in X$ we have $\mathbb{A}\!\mid_x = U(A_x)$, so we can consider the following factorization 
\[\begin{tikzcd}[column sep=large]
	{B} & {\Gamma \mathbb{A}} \\
	{U(A_{q^X_x \phi})} & {U(A_x)}
	\arrow["{\phi}", from=1-1, to=1-2]
	\arrow["{q^X_x}", from=1-2, to=2-2]
	\arrow["{\eta^{A_x}_{q^X_x \phi}}"', from=1-1, to=2-1]
	\arrow["{U(L_{A_x}(q^X_x \phi))}"', from=2-1, to=2-2]
\end{tikzcd}\]
Then the following map 
\[\begin{tikzcd}[row sep=tiny]
	{(X,\tau)} & {(X_B, \tau_B)} \\
	{x} & {\eta^{A_x}_{q^X_x \phi}}
	\arrow["{f}", from=1-1, to=1-2]
	\arrow[from=2-1, to=2-2, shorten <=4pt, shorten >=4pt, maps to]
\end{tikzcd}\]
is continuous. Indeed, for any finitely presented diagonally universal morphism $ n $, we have 
\[ f^{-1}(D_n) = \big{\{} x \in X \mid \, n \leq \eta^{A_x}_{q^X_x \phi}  \big{\}} \]
where the latter condition says that we have a factorization
\[\begin{tikzcd}
	{B} & {\Gamma\mathbb{A}} \\
	{C} & {\phi_*C} & {\mathbb{A}\!\mid_x}
	\arrow["{n}"', from=1-1, to=2-1]
	\arrow["{\phi_*n}" description, from=1-2, to=2-2]
	\arrow["{\phi}", from=1-1, to=1-2]
	\arrow["{n_*\phi}"', from=2-1, to=2-2]
	\arrow["{q^X_x}", from=1-2, to=2-3, curve={height=-6pt}]
	\arrow[from=2-2, to=2-3, dashed, "s"]
\end{tikzcd}\]
But pushouts of finitely presented diagonally universal morphisms are still finitely presented, that is, $ \phi_*n$ is finitely presented as a diagonally universal morphism under $\Gamma \mathbb{A}$: but as $ q^X_x = \colim_{x \in u} q^X_x$ in the slice $ \Gamma \mathbb{A} \downarrow \mathcal{B}$, we have a factorization through some $ q^X_u$ for some $u$ in $\tau$ such that $ x \in u$ 
\[\begin{tikzcd}
	{\Gamma\mathbb{A}} & {\mathbb{A}(u)} \\
	{\phi_*C} & {\mathbb{A}\!\mid_x}
	\arrow["{\phi_*n}"', from=1-1, to=2-1]
	\arrow[from=2-1, to=2-2, "s"]
	\arrow["{q^X_u}", from=1-1, to=1-2]
	\arrow["{q^u_x}", from=1-2, to=2-2]
	\arrow[from=2-1, to=1-2, dashed, "s_u"]
\end{tikzcd}\]
and moreover, any two such factorizations are equalized by a third one: that is, if there are two opens $u,v$ with $ x \in u,v$ such that $s $ factorizes through $q^u_x, \, q^{v}_x$ as $s_u, s_{v}$, then there is some $ w \subseteq u, v$ and $ x \in w$ such that $ s$ factorizes through $w$ as $ q^{w}_x s_w$. 
And moreover, for any $y \in u$, we still have that $ \phi_*n $ factorizes $ q^X_x$, so that \[ n \leq \eta^{A_y}_{q^X_y}\phi \] and hence $ y \in f^{-1}(D_n)$. This means that $u \subseteq f^{-1}(D_n)$. To sum up, we can choose for each $ x \in f^{-1}(D_n)$ an open $ u_x $ with $ x \in u_x$ and $ u_x \subseteq  f^{-1}(D_n)$, and hence we have \[f^{-1}(D_n) = \bigcup_{x \in f^{-1}(D_n)} u_x \] which is hence open in $\tau$. \\

Now we construct a morphism of sheaves $\psi$ as follows. We first construct the direct image part. As seen above, for any $n \in \mathcal{D}_B$, we can exhibit a cover of $ f^{-1}(D_n) $ with a family $ (u_x)_{x \in f^{-1}(D_n)}$ such that $ x \in u_x$ and we have a factorization 
\[\begin{tikzcd}
	{B} & {\Gamma\mathbb{A}} & {} \\
	{C} & {\mathbb{A}(u_x)}
	\arrow["{n}"', from=1-1, to=2-1]
	\arrow["{q^X_{u_x}}", from=1-2, to=2-2]
	\arrow["{\phi}", from=1-1, to=1-2]
	\arrow["{s_x}"', from=2-1, to=2-2]
\end{tikzcd}\] 
Now, as $\mathbb{A}$ is a sheaf, and $ f^{-1}(D_n)$ is open, we have the limit decomposition
\[ \mathbb{A}(f^{-1}(D_n))   = \lim \Big{ ( } 
\begin{tikzcd}
	{\underset{x \in f^{-1}(D_n)}{\prod} \mathbb{A}(u_x)} & {\underset{x, y \in f^{-1}(D_n)}{\prod} \mathbb{A}(u_x \cap u_y)}
	\arrow[from=1-1, to=1-2, shift left=1]
	\arrow[from=1-1, to=1-2, shift right=1]
\end{tikzcd} \Big{)} \]
So we have to check that the data of the arrows $(s_x)_{x \in X}$ define a cone over the diagram above, which amounts to check that for any $x,y \in X$ the following diagram commutes
\[\begin{tikzcd}
	& {\mathbb{A}(u_x)} \\
	{C} && {\mathbb{A}(u_x \cap u_y)} \\
	& {\mathbb{A}(u_y)}
	\arrow["{s_x}", from=2-1, to=1-2]
	\arrow["{q^{u_x}_{u_x \cap u_y}}", from=1-2, to=2-3]
	\arrow["{s_y}"', from=2-1, to=3-2]
	\arrow["{q^{u_y}_{u_x \cap u_y}}"', from=3-2, to=2-3]
\end{tikzcd}\]
For any $z \in u_x \cap u_y$ we have 
\[ q^{u_x}_z s_x n = q^{u_x}_z q^{X}_{u_x} \phi = q^{u_y}_z q^{X}_{u_y} \phi = q^{u_y}_z s_y n \]
hence $s_x$ and $s_y$ produce factorizations of $ q^X_z$, so by what was said before as a consequence of finite presentability of $n$, those two factorization have to be merged into some $v_z \subseteq u_x \cap u_y$: but this is true for any $z \in u_x \cap u_y$, so the opens $v_z$ cover the intersection, and hence the square above commutes. \\

Then the universal property of the limit provides a unique arrow 
\[\begin{tikzcd}
	{B} & {\Gamma\mathbb{A}} \\
	{C} & {\mathbb{A}(f^{-1}(D_n))}
	\arrow["{\phi}", from=1-1, to=1-2]
	\arrow["{\rho^X_{f^{-1}(D_n)}}", from=1-2, to=2-2]
	\arrow["{n}"', from=1-1, to=2-1]
	\arrow["{s_n}"', from=2-1, to=2-2, dashed]
\end{tikzcd}\]
and the data of all those arrows $(s_n: \cod(n) \rightarrow \mathbb{A}(f^{-1}(D_n)))_{n \in \mathcal{D}_B}$ defines a natural transformation $ s : \cod \Rightarrow f_*\mathbb{A}$: now the universal property of the left Kan extension produces a universal arrow as below
\[\begin{tikzcd}[column sep=large]
	{\mathcal{D}_B} & {} & {\mathcal{B}} \\
	{\tau_B^{op}} \\
	{\tau^{op}}
	\arrow["{D}"', from=1-1, to=2-1]
	\arrow["{f^{-1}}"', from=2-1, to=3-1]
	\arrow["{\cod}"{name=0}, from=1-1, to=1-3]
	\arrow["{\mathbb{A}}"{name=1, swap}, from=3-1, to=1-3, curve={height=12pt}]
	\arrow["{\lan_{\!D}\cod}"{name=2, description}, from=2-1, to=1-3]
	\arrow[Rightarrow, "{s}"', from=0, to=1, shift right=8, curve={height=7pt}, shorten <=4pt, shorten >=4pt, crossing over]
	\arrow[Rightarrow, "{\overline{s}}"', from=2, to=1, shift left=1, shorten <=7pt, shorten >=3pt, dashed]
	\arrow[Rightarrow, "{\zeta^B}"', from=0, to=2, shift left=1, shorten <=2pt, shorten >=2pt]
\end{tikzcd}\]
But now, as $\mathbb{A}$ is a sheaf, so is the direct image $ f_*\mathbb{A}$, hence the natural transformation $ \overline{s}$ factorizes through the sheafification of $\lan_{\!D}\cod$, that is
\[\begin{tikzcd}
	{\lan_{\!D}\cod} & {f_*\mathbb{A}} \\
	{\widetilde{B}}
	\arrow["{\overline{s}}", from=1-1, to=1-2]
	\arrow["{\gamma}"', from=1-1, to=2-1]
	\arrow["{\psi^\sharp}"', from=2-1, to=1-2]
\end{tikzcd}\]
This returns the desired direct image part $\psi^\sharp$. \\

For the inverse image part, the adjunction $ f^*\dashv f_*$ associates a unique mate $ \psi^\flat$ to $ \psi^\sharp$. Now we want to check that this mate behaves correctly at the stalks, that is, that $\psi^\flat = U( L_{A_x}(q^X_x \phi))$. For any $x \in X$ and any $ n \in \mathcal{D}_B$, we have from the remark in part 1 an commutative square 
\[\begin{tikzcd}
	{\widetilde{B}(D_n)} & {\mathbb{A}(f^{-1}(D_n))} \\
	{\widetilde{B}\!\mid_{f(x)}} & {\mathbb{A}\!\mid_x}
	\arrow["{\rho^{D_n}_{f(x)}}"', from=1-1, to=2-1]
	\arrow["{\psi^\sharp_{D_n}}", from=1-1, to=1-2]
	\arrow["{\rho^{f^{-1}(D_n)}_x}", from=1-2, to=2-2]
	\arrow["{\psi^\flat_x}"', from=2-1, to=2-2]
\end{tikzcd}\]
But the stalk of $\widetilde{B}$ at $f(x) = \eta^{A_x}_{\rho^X_x \phi}$ is both the unit of the factorization of $ \rho^X_x\phi$ and the filtered colimit 
\[ \widetilde{B}\!\mid_{f(x)} = \underset{n \leq \eta^{A_x}_{\rho^X_x \phi}}{\colim} \; \widetilde{B}(D_n) \]
inducing $ \phi^\flat_x$ as the universal map \[ \langle \rho^{f^{-1}(D_n)}_x \psi^{\sharp}_{D_n} \rangle_{n \leq \eta^{A_x}_{\rho^X_x \phi}} \]
But from Diers condition, this later must be the left part of the factorization, so that we have the desired equality. Now  we gather all the maps $ (L_{A_x}(q^X_x \phi))_{x \in X}$ to complete the data of the morphism of $U$-space 
\[\begin{tikzcd}[column sep=large]
	{ ((X_B, \tau_B), \widetilde{B}, (A_x)_{x \in X_B})} & {} & {((X,\tau), \mathbb{A},(A_x)_{x \in X})}
	\arrow["{(f, \psi, (L_{A_x}(\rho^{X}_x\phi))_{x \in X})}", from=1-1, to=1-3]
\end{tikzcd}\]
\end{proof}

\begin{remark}
While suited for a large range of examples where the spectrum is expected to be spatial, this construction of the spectrum as a topological space causes some loss of information as it equips the set of local units with a structure of poset, mimicking the specialization order of the spectral topology, while in general local units under a given object and diagonally universal morphisms between them form a category. In some sense Diers only considers the locallic reflection of a small site made of finitely presentable diagonally universal morphisms, though the topos associated to this site may not be locallic, for instance in the case of the \'{e}tale spectrum of a ring. This problem would be fixed by introducing a variant of Diers's construction in terms of ionad, where the link with the other notions of spectrum would appear more clearly; as this would require more involved notion of toposes and ionads, we chose to postpone this task to a later paper devoted to the synthesis between different approaches of the construction of spectra from more topos-theoretic point of view. 
\end{remark}

\section{Spectrum of an arbitrary $\mathcal{B}$-space}

In the previous section we recalled Diers original construction of the spectrum of an object of $\mathcal{B}$ for a Diers context. Observe that the spectral functor thus defines is left adjoint to a global section functor sending a $U$-space to the $\mathcal{B}$-object of global sections of its structural sheaf. This process is somewhat reminiscent of the relative adjoint $ \mathcal{B} \rightarrow \Pi \mathcal{A}$ constructed in \cite{partI}[Proposition 2.6], which returns the points of the spectrum of a given object. Hence the fact that this relative adjoint could extend into a functor $ \Pi \mathcal{B} \rightarrow \Pi\mathcal{A}$ left adjoint to $\Pi U$ invites us to extend Diers Spectral functor to a larger class of $\mathcal{B}$-spaces, amongst which the objects of $\mathcal{B}$ embed as $\mathcal{B}$-spaces with the point as underlying space. This generalized spectrum will be left adjoint to the a functor sending a $U$-space to the corresponding $\mathcal{B}$-space one gets by applying stalk-wise the functor $U$. The process in this part is also totally point-set, hence both quite concrete and somewhat ``handmade", in contrast to the more abstract, yet purer point-free approach in the topos-theoretic methods of \cite{survey}. \\

Throughout this section, we fix a right multi-adjoint $U : \mathcal{A} \rightarrow \mathcal{B}$ satisfying Diers conditions. 

\begin{definition}
We define a \emph{$\mathcal{B}$-space} as the data of a topological space $ (X,\tau)$ together with a sheaf of $\mathcal{B}$-object $ \mathbb{B} : \tau^{op} \rightarrow \mathcal{B}$. A \emph{morphism of $\mathcal{B}$-spaces} $ ((X_1,\tau_1), \mathbb{B}_1) \rightarrow ((X_2,\tau_2), \mathbb{B}_2)$ is the data of an continuous map $ f: X_2 \rightarrow X_1 $ and a pair of morphisms of sheaves \[ (\phi^\flat : f^*\mathbb{B}_1 \rightarrow \mathbb{B}_2, \, \phi^\sharp : \mathbb{B}_1 \rightarrow f_*\mathbb{B}_2)\] corresponding through the adjunction $ f^* \dashv f_*$.
We denote as $ \mathcal{B}-Spaces$ the category of $ \mathcal{B}$-spaces and morphisms between them.
\end{definition}

\begin{proposition}
There is a functor 
\[ U-Spaces \stackrel{\iota_u}{\longrightarrow} \mathcal{B}-Spaces \]
sending a $U$-Space $ ((X, \tau), \mathbb{A}, (A_x)_{x \in X})$ to the $\mathcal{B}$-space $ ((X,\tau), \mathbb{A})$, and a morphism of $ U$-spaces to the induced morphism of $\mathbb{B}$-space. 
\end{proposition}

\begin{remark}
Beware that, without additional assumption on $U$, this functor may not be faithful, nor even injective on objects. Its action on objects is to forget which $ A_x$ of $\mathcal{A}$ the stalk of $ \mathbb{A}\mid_x = U(A_x)$ comes from. In some sense, the structural sheaf of $ \iota_U((X, \tau), \mathbb{A}, (A_x)_{x \in X})$ has its stalks in the essential image of $U$, so that two $U$-spaces over a same topological space $((X, \tau), \mathbb{A}_1, (A^1_x)_{x \in X})$ and $ ((X, \tau), \mathbb{A}_2, (A^2_x)_{x \in X})$ may become isomorphic in $\mathcal{B}-Spaces$ if $ U(A^1_x) \simeq U(A^2_x)$ in each $ x$, though some $ A^1_x$ and $ A^2_x$ were not isomorphic in $\mathcal{A}$. The same phenomena happens for arrows. \\

Moreover, in most cases, $\iota_U$ is not full, unless $ U$ was itself full, but this corresponds to a specific kind of geometries returning $ T_1$ spaces. In general we do not desire the functor $\iota_U$ to be full as the morphisms $ (f, (\phi^\flat, \phi^\sharp), (u_x)_{x \in X})$ in $U-Spaces$ have the additional data $(u_x)_{x \in X}$ attached to the $\phi^\flat$ part of the morphism of sheaf, forcing the map at stalks $ \phi^\flat_ x : \mathbb{A}_1\mid_{f(x)} \rightarrow \mathbb{A}_2\mid_x$ to be equal to $ U(u_x) : U(A^1_{f(x)}) \rightarrow U(A_x^2)$, hence in the range of $U$, while there is no such condition for general morphisms of $\mathcal{B}$-spaces.  
\end{remark}

In the following we fix a $ \mathcal{B}$-space $ ((X,\tau), \mathbb{B})$. We are going to process by gluing the spectra of the stalks of the structural sheaf $ \mathbb{B}$ at points of $X$, and equip it with a topology generated from the finitely presented diagonally universal morphisms under the values of the sheaf $ \mathbb{B}$ at open of $\tau$. 

\begin{definition}\label{Dierspacegeneralcase}
We define the \emph{Diers space} of $ ((X,\tau), \mathbb{B})$ as the space $ (X_\mathbb{B}, \tau_\mathbb{B})$ with
\[ X_\mathbb{B} = \big{ \{ } (x,\xi) \, \mid \; x \in X, \; \xi \in X_{\mathbb{B}\mid_x} \big{ \} } \]
equiped with the topology generated as \[ \tau_\mathbb{B} = \langle D_{(u,n)} \rangle_{(u,n) \in \mathcal{D}_\mathbb{B}}\] where $ \mathcal{D}_\mathbb{B} = \{ u \in \tau, \; n \in \mathcal{D}_{\mathbb{B}(u)} \} $ and
\[ D_{(u,n)} = \big{\{} (x, \xi) \, \mid \; x \in u, \; {q^u_x}_*n \leq \xi \big{ \}}  
\]
\end{definition}
\begin{remark}
The condition defining the basic open set $ D_{(u,n)}$ concerns the finitely presented diagonally universal morphism $n$ under $ \mathbb{B}(u)$. Each $ \mathbb{B}(u)$ is an object in $\mathcal{B}$, as well as each stalk $ \mathbb{B}$, which is the filtered colimit of all $ \mathbb{B}(u)$ such that $ x \in u$. Hence for such neighborhood $u$ there is a canonical inclusion $ q^u_x : \mathbb{B}(u) \rightarrow \mathbb{B}\mid_x$, and we can push $n$ along this inclusion to get the finitely presented diagonally universal morphism $ {q^u_x}_*n $ in $\mathcal{D}_{\mathbb{B}\mid_x}$. The condition then requires that we have a factorization\[\begin{tikzcd}
	{\mathbb{B}(u)} & {C} \\
	{\mathbb{B}\!\mid_x} & {{q^u_x}_*C} \\
	& {U(A_\xi)}
	\arrow["{q^u_x}"', from=1-1, to=2-1]
	\arrow["{{q^u_x}_*n}", from=2-1, to=2-2]
	\arrow["{n}", from=1-1, to=1-2]
	\arrow[from=1-2, to=2-2, dashed]
	\arrow["\lrcorner"{very near start, rotate=180}, from=2-2, to=1-1, phantom]
	\arrow["{\xi}"', from=2-1, to=3-2]
	\arrow[from=2-2, to=3-2]
\end{tikzcd} \]
\end{remark}

\begin{remark}
Also it is interesting to note that, if $ \tau$ admits a basis $ \tau_0 \hookrightarrow \tau$, then one can generate $\tau_{\mathcal{B}}$ from the open $ D_{(u,n)}$ with $ u \in \tau_0$. While this is not that important in this point-set context as $\tau$ has anyway to be a small set, the use of a basis is essential if we want to extend such a construction to $\mathbb{B}$ objects in arbitrary toposes as in the topos theoretic approach. But we remain in a point-set spirit within this paper.
\end{remark}

It is worth checking that this is a basis:

\begin{lemma}
For each $(u,n)$ and $ (v,m)$ in $\mathbb{D}_\mathbb{B}$, we have \[ D_{(u,n)} \cap D_{(v,m)} = D_{(u\cap v, {\rho^u_{u \cap v}}_*n \vee
{\rho^v_{u \cap v}}_*m)}\]
\end{lemma}

\begin{proof}
First, it is clear that if $ (x, \xi) \in D_{(u,n)} \cap D_{(v,m)}$ then $ x \in u \cap v$. Then the projections $ \rho^u_x$ and $ \rho^v_x$ factorize as
\[\begin{tikzcd}
	{\mathbb{B}(u)} && {\mathbb{B}(v)} \\
	& {\mathbb{B}(u \cap v)} \\
	& {\mathbb{B}\!\mid_x}
	\arrow["{\rho^u_{u \cap v}}", from=1-1, to=2-2]
	\arrow["{\rho^v_{u \cap v}}"', from=1-3, to=2-2]
	\arrow["{\rho^{u \cap v}_x}" description, from=2-2, to=3-2]
	\arrow["{\rho^u_x}"', from=1-1, to=3-2, curve={height=12pt}]
	\arrow["{\rho^u_x}", from=1-3, to=3-2, curve={height=-12pt}]
\end{tikzcd}\]
Now by left cancellation of pushouts we have that \[ {\rho^u_x}_*n = {\rho^{u \cap v}_x}_*({\rho^u_{u \cap v}}_*n) \textrm{  and  } {\rho^v_x}_*m = {\rho^{v \cap v}_x}_*({\rho^v_{u \cap v}}_*m) \] as depicted in the following diagram 
\[\begin{tikzcd}[sep=large]
	& {\mathbb{B}(u)} && {\mathbb{B}(v)} \\
	{C} && {\mathbb{B}(u \cap v)} && {D} \\
	& {{\rho^u_{u \cap v}}_*C} & {\mathbb{B}\!\mid_x} & {{\rho^v_{u \cap v}}_*D} \\
	& {{\rho^u_x}_*C} && {{\rho^v_x}_*D} \\
	&& {A_\xi}
	\arrow["{\rho^u_{u \cap v}}", from=1-2, to=2-3]
	\arrow["{\rho^v_{u \cap v}}"', from=1-4, to=2-3]
	\arrow["{\rho^{u \cap v}_x}" description, from=2-3, to=3-3]
	\arrow["{\rho^u_x}"'{near start}, from=1-2, to=3-3, curve={height=12pt}]
	\arrow["{\rho^v_x}"{near start}, from=1-4, to=3-3, curve={height=-12pt}]
	\arrow["{n}"', from=1-2, to=2-1]
	\arrow["{{\rho^u_{u \cap v}}_*n}" description, from=2-3, to=3-2]
	\arrow["{{\rho^u_x}_*n}" description, from=3-3, to=4-2]
	\arrow["{{\rho^v_x}_*m}" description, from=3-3, to=4-4]
	\arrow["{{\rho^v_{u \cap v}}_*m}" description, from=2-3, to=3-4]
	\arrow["{m}", from=1-4, to=2-5]
	\arrow[from=2-1, to=3-2]
	\arrow[from=3-2, to=4-2]
	\arrow[from=2-5, to=3-4]
	\arrow[from=3-4, to=4-4]
	\arrow["{\xi}" description, from=3-3, to=5-3]
	\arrow[from=4-2, to=5-3, dashed]
	\arrow[from=4-4, to=5-3, dashed]
	\arrow[from=2-1, to=4-2, curve={height=12pt}]
	\arrow["\lrcorner"{very near start, rotate=135}, from=3-2, to=1-2, phantom]
	\arrow["\lrcorner"{very near start, rotate=135}, from=3-4, to=1-4, phantom]
	\arrow["\lrcorner"{very near start, rotate=90}, from=4-2, to=2-3, phantom]
	\arrow["\lrcorner"{very near start, rotate=180}, from=4-4, to=2-3, phantom]
	\arrow[from=2-5, to=4-4, curve={height=-12pt}]
\end{tikzcd}\]
Moreover commutation of pushouts ensures that all squares in the following cube are pushouts
\[\begin{tikzcd}
	& {\mathbb{B}(u \cap v)} \\
	{{\rho^u_{u \cap v}}_*C} & {\mathbb{B}\!\mid_x} & {{\rho^v_{u \cap v}}_*D} \\
	{{\rho^u_x}_*C} & {{\rho^u_{u \cap v}}_*C +_{\mathbb{B}(u \cap v)} {\rho^v_{u \cap v}}_*D} & {{\rho^v_x}_*D} \\
	& {{\rho^u_x}_*C  +_{\mathbb{B}\mid_x}{\rho^u_x}_*C}
	\arrow["{\rho^{u \cap v}_x}" description, from=1-2, to=2-2]
	\arrow["{{\rho^u_{u \cap v}}_*n}" description, from=1-2, to=2-1]
	\arrow["{{\rho^u_x}_*n}"'{near start} description, from=2-2, to=3-1]
	\arrow["{{\rho^v_x}_*m}"{near start} description, from=2-2, to=3-3]
	\arrow["{{\rho^v_{u \cap v}}_*m}" description, from=1-2, to=2-3]
	\arrow[from=2-1, to=3-1]
	\arrow[from=2-3, to=3-3]
	\arrow[from=3-1, to=4-2, dashed]
	\arrow[from=3-3, to=4-2, dashed]
	\arrow["\lrcorner"{very near start, rotate=90}, from=3-1, to=1-2, phantom]
	\arrow["\lrcorner"{very near start, rotate=180}, from=3-3, to=1-2, phantom]
	\arrow[from=2-1, to=3-2, crossing over]
	\arrow[from=2-3, to=3-2, crossing over]
	\arrow[from=3-2, to=4-2]
	\arrow["\lrcorner"{very near start, rotate=135}, from=3-2, to=1-2, phantom]
	\arrow["\lrcorner"{very near start, rotate=180}, from=4-2, to=2-1, phantom]
	\arrow["\lrcorner"{very near start, rotate=90}, from=4-2, to=2-3, phantom]
\end{tikzcd}\]
But we have \[  {\rho^u_x}_*n \vee
{\rho^v_x}_*m  \leq \xi  \] 
and hence \[ {\rho^u_{u \cap v}}_*( {\rho^u_{u \cap v}}_*n \vee
{\rho^v_{u \cap v}}_*m) \leq \xi \]
which ensures that $ (x,\xi) \in D_{(u\cap v, {\rho^u_{u \cap v}}_*n \vee
{\rho^v_{u \cap v}}_*m)}$
\end{proof}

\begin{proposition}
We have a continuous map
\[\begin{tikzcd}[row sep=tiny]
	{(X_\mathbb{B}, \tau_\mathbb{B})} & {(X, \tau)} \\
	{(x,\xi)} & {x}
	\arrow["{\eta}", from=1-1, to=1-2]
	\arrow[from=2-1, to=2-2, maps to]
\end{tikzcd}\]
which is moreover open. \end{proposition}

\begin{proof}
For $ u \in \tau$, we have $ \eta^{-1}(u) = \{ (x,\xi) \mid x \in u, \, \xi \in X_{\mathbb{B}\mid_x} \} $ which coincides with the basic open $D_{(u, 1_{\mathbb{B}(u)})}$ as $ {q^u_x}_*(1_{\mathbb{B}(u)}) = 1_{\mathbb{B}\mid_x}$ is the initial object of $ \mathcal{D}_{\mathbb{B}\!\mid_x}$ hence factorizes any local unit $\xi $ under $ \mathbb{B}\mid_x$. It is moreover open because the direct image of an open $D_{(u,n)} $ along $\eta$ is the underlying open of $\tau$, that is $ \eta(D_{(u,n)}) = u$.
\end{proof}

\begin{proposition}
and also in each $ x \in X$ a continuous map
\[\begin{tikzcd}[row sep=tiny]
	{(X_{\mathbb{B}\mid_x}, \tau_{\mathbb{B}\mid_x})} & {(X_\mathbb{B}, \tau_\mathbb{B})} \\
	{\xi} & {(x,\xi)}
	\arrow["{\iota_x}", from=1-1, to=1-2, hook]
	\arrow[from=2-1, to=2-2, maps to]
\end{tikzcd}\]
exhibiting the Diers space as the coproduct in $\mathcal{T}op$ of the Diers spaces of the stalks \[ X_\mathbb{B} = \coprod_{x \in X} X_{\mathbb{B}\mid_x} \]
and $\eta$ as the induced map $ \langle c_x \rangle_{x \in X}$ where $ c_x : X_{\mathbb{B}\mid_x} \rightarrow X$ is the constant map equal to $x$. 
\end{proposition}

\begin{proof}
Observe that if $x \in u$ we have $\iota_x^{-1}(D_{(u,n)}) = D_{{q^u_x}_*n}$ in $ \tau_{\mathbb{B}\mid_x}$ by property of the pushout; otherwise we have $ \iota_x^{-1}(u)=0$ if $x \notin u$. 
\end{proof}

At each $ u \in \tau$ we have an open $ D_{(u, 1_{\mathbb{B}(u)})}$, and this define an open subspace 
\[\begin{tikzcd}
	{(D_{(u, 1_{\mathbb{B}(u)})}, \tau_{\mathbb{B}}\mid_{D_{(u, 1_{\mathbb{B}(u)}}})} && {(X_\mathbb{B}, \tau_\mathbb{B})} 
	\arrow["{\iota_{(u, 1_{\mathbb{B}(u)})}}", from=1-1, to=1-3, hook]
\end{tikzcd}\]

Observe that we also have, for any $u \in \tau$, a canonical map 
\[\begin{tikzcd}[row sep=tiny]
	D_{(u, 1_{\mathbb{B}(u)})} & {X_{\mathbb{B}(u)}} \\
	{(x,\xi)} & {\eta^{A_\xi}_{\xi \rho^u_x}}
	\arrow["{p_u}", from=1-1, to=1-2]
	\arrow[from=2-1, to=2-2, maps to]
\end{tikzcd}\]sending any $ (x,\xi)$ with $x \in u$ to the unit under $\mathbb{B}(u)$ factorizing the composite
\[\begin{tikzcd}
	{\mathbb{B}(u)} & {U(A_{\xi\rho^u_x})} \\
	{\mathbb{B}\!\mid_x} & {U(A_\xi)}
	\arrow["{\eta^{A_\xi}_{\xi\rho^u_x}}", from=1-1, to=1-2]
	\arrow["{U(L_{A_\xi}(\xi\rho^u_x))}", from=1-2, to=2-2]
	\arrow["{\xi}"', from=2-1, to=2-2]
	\arrow["{\rho^u_x}"', from=1-1, to=2-1]
\end{tikzcd}\]
Moreover Diers condition ensures that this map is continuous for $ U(A_{\xi\rho^u_x}) = \colim_{n \leq \xi\rho^u_x} \cod(n)$ and $n \leq \xi\rho^u_x$ if and only if $ {\rho^{u}_x}_*n \leq \xi$. Hence in particular $ p_u^{-1}$ coincides on the basis $ \mathcal{D}_{\mathbb{B}(u)}$ with the pushout along $ {\rho^u_x}$, that is we have 
\[\begin{tikzcd}[row sep=tiny]
	{\mathcal{D}_{\mathbb{B}(u)}} & {\mathcal{D}_{\mathbb{B}}} \\
	{n} & {(u,n)}
	\arrow["{p_u^{-1}}", from=1-1, to=1-2, hook]
	\arrow[from=2-1, to=2-2, shorten <=2pt, shorten >=2pt, maps to]
\end{tikzcd}\]

Before going further, we need the following key observation:

\begin{lemma}
For any $x \in X$ and any finitely presented diagonally universal morphism $ n$ in $ \mathcal{D}_{\mathbb{B}\mid_x} $, there is some $u$ with $ x \in u$ and $m $ in $\mathcal{D}_{\mathbb{B}(u)}$ such that $ n= {\rho^u_x}_*m$. 
\end{lemma}

\begin{proof}
Recall from \cite{partI} that diagonally universal morphisms of finite presentation under $\mathbb{B}\!\mid_x$ are those induced as pushouts 
\[\begin{tikzcd}
	{K} & {K'} \\
	{\mathbb{B}\!\mid_x} & {C}
	\arrow["{a}"', from=1-1, to=2-1]
	\arrow["{l}", from=1-1, to=1-2]
	\arrow["{l_*a}", from=1-2, to=2-2]
	\arrow["{n}"', from=2-1, to=2-2]
	\arrow["\lrcorner"{very near start, rotate=180}, from=2-2, to=1-1, phantom]
\end{tikzcd}\]
whith $ K, \, K'$ in $\mathcal{B}_\omega$ and $ l$ diagonally universal. But now, as we have a filtered colimit decomposition $ \mathbb{B}\!\mid_x = \colim_{x \in u} \, \mathbb{B}(u)$ and $K$ is finitely presented, $a $ factorizes through some $ \rho^u_x$, and by left cancellation of pushouts, the front square in the following diagram is a pushout for the back and top are so
\[\begin{tikzcd}
	{K} && {K'} \\
	& {\mathbb{B}(u)} && {b_*K'} \\
	& {\mathbb{B}\!\mid_x} && {C}
	\arrow["{a}"', from=1-1, to=3-2, curve={height=12pt}]
	\arrow["{\rho^u_x}"', from=2-2, to=3-2]
	\arrow[from=1-1, to=2-2]
	\arrow["{l}", from=1-1, to=1-3]
	\arrow["{b_*l}" description, from=2-2, to=2-4, crossing over]
	\arrow["{n}"', from=3-2, to=3-4]
	\arrow[from=1-3, to=2-4]
	\arrow[from=2-4, to=3-4]
	\arrow[from=1-3, to=3-4, curve={height=12pt}]
	\arrow["\lrcorner"{very near start, rotate=180}, from=3-4, to=1-1, shift left=2, phantom]
	\arrow["\lrcorner"{very near start}, from=2-4, to=1-1, phantom]
\end{tikzcd}\]
so that $ n = {\rho^u_x}_*b_l$, but $ b_*l $ is in $^\mathcal{D}_{\mathbb{B}(u)}$.
\end{proof}

In some sense, this result says that the categories $ \mathcal{D}_{\mathbb{B}(u)}$ for $x \in u$ are jointly essentially surjective on $ \mathcal{D}_{\mathbb{B}\mid_x}$. This has also the following consequence which will be central to control the stalks of the structural sheaf we are going to construct: 

\begin{theorem}\label{stalk lemma}
For any $ (x, \xi) \in X_\mathbb{B}$, we have a filtered colimit decomposition of $\xi$ in the slice $ \overrightarrow{\mathcal{B}}$:
\[ \xi = \underset{x \in u}{\colim}\, {\rho^u_x}_*\eta^{A_\xi}_{{\xi}\rho^u_x} \]
As a consequence, the cocone $ (U(L_{A_\xi}({\xi}\rho^u_x)))_{x \in u}$ exhibits $ U(A_\xi)$ as a filtered colimit 
\[ U(A_\xi) \simeq \underset{x \in u}{\colim}\, {\rho^u_x}_*U(A_{\xi\rho^u_x}) \]
\end{theorem}

\begin{proof}
First, for any $u $ with $ x \in u$, we have $ n \leq \eta^{A_\xi}_{A_{\xi\rho^u_x}} $ if and only if $ {\rho^u_x}_*n \leq \xi$. But Diers condition at $ \mathbb{B}\!\mid_x$ says that $ \xi = \colim_{n \leq \xi} \, n$, and from the previous lemma, we know that we can precise this colimit as 
\begin{align*}
    \xi &\simeq  \underset{ x \in u \atop n \in \mathcal{D}_{\mathbb{B}(u)} \; n \leq \eta^{A_\xi}_{A_{\xi\rho^u_x}} }{\colim} \, {\rho^u_x}_*n \, \\
    &\simeq \, \underset{x \in u}{\colim} \; {\rho^u_x}_* \big{(}\underset{n \leq \eta^{A_\xi}_{A_{\xi\rho^u_x}}}{\colim} \, \cod(n) \big{)} \, \\
    &\simeq \,  \underset{x \in u}{\colim}\, {\rho^u_x}_*\eta^{A_\xi}_{\xi\rho^u_x}
\end{align*}

where the last isomorphism comes from Diers condition at $ \eta^{A_\xi}_{\xi \rho^u_x}$. The second item is deduced from the fact that $\cod$ commutes with filtered colimits.
\end{proof}

Now we describe the process to construct a structural sheaf $ \widetilde{\mathbb{B}}$ on ${(X_\mathbb{B}, \tau_\mathbb{B})}$ associated to $\mathbb{B}$. The poset $ \mathcal{D}_{\mathbb{B}} = \{ D_{(u,n)} \mid \, u \in \tau, \, n \in \mathcal{D}_{\mathbb{B}(u)} \}$ is equiped with a functor 
\[\begin{tikzcd}[row sep=tiny]
	{ \mathcal{D}_{\mathbb{B}}} & {\mathcal{B}} \\
	{(u,n)} & {\cod(n)}
	\arrow["{\cod \pi_1}", from=1-1, to=1-2]
	\arrow[from=2-1, to=2-2, maps to]
\end{tikzcd}\]
and then consider its Kan extension along the inclusion into the spectral topology
\[\begin{tikzcd}[sep=large]
	{ \mathcal{D}_{\mathbb{B}}} & {\mathcal{B}} \\
	{\tau_\mathbb{B}}
	\arrow["{\cod \pi_1}"{name=0}, from=1-1, to=1-2]
	\arrow["{\iota_\mathbb{B}}"', from=1-1, to=2-1]
	\arrow["{\lan_{\iota_\mathbb{B}} \cod \pi_1}"{name=1, swap}, from=2-1, to=1-2]
	\arrow[Rightarrow, "{\zeta^{\mathbb{B}}}"'{very near start}, from=0, to=1, shorten <=-4pt, shorten >=13pt, shift right=5]
\end{tikzcd}\]
and define $ \widetilde{\mathbb{B}}$ as the sheafification for the topology $ \tau_\mathbb{B}$
\[ \widetilde{\mathbb{B}} = \mathfrak{a}_{\tau_\mathbb{B}}(\lan_{\!\iota_\mathbb{B} } \cod \pi_1) \]

At this point it is worth giving some detail on the behavior on the sheaf $\mathbb{B}$ and its relation with the spectra of the values at open $ \mathbb{B}(u)$. For each $ u \in \tau$ the object $ \mathbb{B}(u)$ is in $\mathcal{B}$, hence has itself a spectrum \[ (X_{\mathbb{B}(u)}, \tau_{\mathbb{B}(u)}, \widetilde{\mathbb{B}(u)})\]
Then the structural sheaf $ \widetilde{\mathbb{B}}$ can be compared to the structural sheaves of the form $ \widetilde{\mathbb{B}(u)}$ as follows: indeed we had in each $u \in \tau$ an inclusion on one side along which one can restrict the structural sheaf $\widetilde{\mathbb{B}}$ into $ \iota_{(u,1_{\mathbb{B}(u)})}^*\widetilde{\mathbb{B}}$. But after restriction we could define a continuous map $p_u$
\[\begin{tikzcd}
	{	{(D_{(u, 1_{\mathbb{B}(u)})}, \tau_{\mathbb{B}}\mid_{D_{(u, 1_{\mathbb{B}(u)})}})}} & {(X_{\mathbb{B}(u)}, \tau_{\mathbb{B}(u)})} \\
	{{(X_\mathbb{B}, \tau_\mathbb{B})}}
	\arrow["{\iota_{(u,1_{\mathbb{B}(u)})} }"', from=1-1, to=2-1, hook]
	\arrow["{p_u}", from=1-1, to=1-2]
\end{tikzcd}\]
\begin{lemma}\label{restrictions}
For each $ u \in \tau$, the restriction of $\widetilde{\mathbb{B}}$ along $p_u$ is related to $ \widetilde{\mathbb{B}(u)}$ through a canonical morphism of sheaf
\[ \widetilde{p_u}^\sharp : \widetilde{\mathbb{B}(u)} \rightarrow  {p_u}_*\widetilde{\mathbb{B}}\mid_{D_{(u, 1_{\mathbb{B}(u)})}} \]
\end{lemma}

\begin{proof}
This can be tested open-wisely on the basis for each $ n \in \mathcal{D}_{\mathbb{B}(u)}$, and at the level of the structural presheaves: indeed we have $p_u^{-1}(D_n) = D_{(u,n)}$ and we have 
\[ \overline{\mathbb{B}(u)}(D_n) =  \underset{D_{n} \subseteq D_m}{\colim } \cod(m)\]
But observe that $D_n \subseteq D_m$ implies that $ D_{(u,n)} \subseteq D_{(u,m)}$: indeed, the first condition says that for any $ \zeta \in X_{\mathbb{B}(u)}$, $ n \leq \zeta$ if and only if $ m \leq \zeta$; but then, as for any $ x \in u$ and any $ \xi \in X_{\mathbb{B}\mid_x}$, by \cref{stalk lemma} we have that $\xi = \colim_{x \in v} {\rho^v_x}_*U(A_{\xi \rho^v_x}) $, then $n \leq \eta^{A_\xi}_{\xi \rho^u_x}$ implies $m \leq \eta^{A_\xi}_{\xi \rho^u_x}$, but this condition also means that $ {\rho^u_x}_*n \leq \xi $ implies $ {\rho^u_x}_*m \leq \xi $. Hence the implication. But now recall that the restriction of $\overline{\mathbb{B}} $ has the same values as $ \overline{\mathbb{B}}$ itself for opens of $ D_{(u, 1_{\mathbb{B}(u)}) }$. But $  \overline{\mathbb{B}}$ is itself computed as 
\[  \overline{\mathbb{B}}(D_{(u,n)}) = \underset{D_{(u,n)} \subseteq D_{(v,m)}}{\colim } \cod(m)\]  
Hence the first colimit ranges over a subset of the indexing set of the second colimit, inducing a canonical arrow 
\[  \overline{\mathbb{B}(u)}(D_n) \stackrel{(\overline{p_u})_n}{\longrightarrow} \overline{\mathbb{B}}(D_{(u,n)})  \]
and those arrows, being induced from universal properties, define altogether a natural transformation. Now the desired $ \widetilde{p_u}^\sharp$ is the induced morphism of sheaves after sheafification. 
\end{proof}

Now let us look at the local behavior of the structural sheaf. 

\begin{proposition}
The stalk of $\widetilde{ \mathbb{B}}$ at $ (x,\xi)$ is given as
\[ \widetilde{ \mathbb{B}}\!\mid_{(x,\xi)} = U(A_\xi) \]
\end{proposition}

\begin{proof}
We just have to compute the stalk of the structural presheaf $ \overline{\mathbb{B}}$ at $(x, \xi)$: 
\[ \overline{\mathbb{B}}\!\mid_{(x, \xi)} \simeq \underset{x \in u \atop {\rho^u_x}_*m \leq \xi}{\colim}\, \overline{\mathbb{B}}(D_{(u,n)}) = \underset{x \in u \atop {\rho^u_x}_*m \leq \xi}{\colim}\, \underset{D_{(u,n)} \subseteq D_{(v,m)}}{\colim} \cod(m) \]
But we have that the set of all the $ D(u, {\rho^u_v}_*m)$ for $u \subseteq v$ and $ m \in \mathcal{D}_{\mathbb{B}(v)}$ is cofinal in the indexing set of the inner colimit, so we have 
\[ \overline{\mathbb{B}}\!\mid_{(x, \xi)} \simeq \,  \underset{x \in u \atop {\rho^u_x}_*m \leq \xi}{\colim}\, \underset{D_{(u,n)} \subseteq D_{(u,{\rho^v_u}_*m)} \atop u \leq v, \;m \in \mathcal{D}_{\mathbb{B}(v)}} {\colim}\, \cod({\rho^u_x}_*m) \]
But in one hand we have $ \cod({\rho^u_x}_*m) = {\rho^u_x}_*(\cod(m))$, and in the other hand, the indexing set above is cofinal in the set of all $D_{(u,n)}$ with $ x \in u$ and ${\rho^u_x}_*n \leq \xi $: this entails that
\[ \overline{\mathbb{B}}\!\mid_{(x, \xi)} \simeq \underset{x \in u}{\colim}\, \underset{{\rho^u_x}_*m \leq \xi}{\colim} \, {\rho^u_x}_*(\cod(m)) \]
But by \cref{stalk lemma}, we know that this later colimit is $ U(A_\xi)$. 
\end{proof}

\begin{remark}
To come back to the comparison morphism $ \widetilde{p_u}^\sharp$, we can see now that the corresponding inverse image part $ \widetilde{p_u}^\flat$ is obtained from the universal property of the colimit in the square 
\[\begin{tikzcd}
	{\widetilde {\mathbb{B}(u)}(D_n)} & {\widetilde{\mathbb{B}}(D_{(u,n)})} \\
	{U(A_{\xi \rho^u_x})} & {U(A_\xi)}
	\arrow["{\eta^{A_\xi}_{\xi \rho^u_x}}"', from=1-1, to=2-1]
	\arrow["{\widetilde{p_u}^\flat_\xi \atop= U(L_{A_\xi}(\xi \rho^u_x))}"', from=2-1, to=2-2]
	\arrow["{\widetilde{p_u}^\sharp_n}", from=1-1, to=1-2]
	\arrow["{\rho^{D_{(u,n)}}_\xi}", from=1-2, to=2-2]
\end{tikzcd}\]
\end{remark}

Now we can also compare the structural sheaf $ \widetilde{\mathbb{B}}$ and the structural sheaves $ \widetilde{\mathbb{B}\!\mid_x}$:

\begin{lemma}\label{stalkrestriction}
For each $ (x, \xi) \in X_\mathbb{B}$, we have an isomorphism
\[ \iota_x^*\widetilde{\mathbb{B}} \simeq \widetilde{\mathbb{B}\!\mid_x} \]
\end{lemma}
\begin{proof}
For each $ u \in \tau$ such that $ x \in u$ and each $n \in \mathcal{D}_{\mathbb{B}(u)}$ we have a canonical map \[ \sigma_{(u,n)}: \cod(n) \rightarrow \widetilde{\mathbb{B}\!\mid_x}(D_{{\rho^u_x}_*n})\] provided by the composite of the right vertical maps 
\[\begin{tikzcd}
	{\mathbb{B}(u)} & {C} \\
	{\mathbb{B}\!\mid_x} & {{\rho^u_x}_*C} \\
	& {\widetilde{\mathbb{B}\!\mid_x}(D_{{\rho^u_x}_*n})}
	\arrow["{\rho^u_x}"', from=1-1, to=2-1]
	\arrow["{{\rho^u_x}_*n}", from=2-1, to=2-2]
	\arrow["{n}", from=1-1, to=1-2]
	\arrow[from=1-2, to=2-2]
	\arrow["\lrcorner"{very near start, rotate=180}, from=2-2, to=1-1, phantom]
	\arrow[from=2-1, to=3-2]
	\arrow["{\gamma_{{\rho^u_x}_*n}\zeta^{\mathbb{B}\!\mid_x}_{{\rho^u_x}_*n}}", from=2-2, to=3-2]
\end{tikzcd}\]
In the case where $ x \notin u $, then $\iota_x^{-1}(u)=\emptyset$ and then $ \widetilde{\mathbb{B}\!\mid_x}(\emptyset)=1$ is always the terminal object of $\mathcal{B}$: then for each $ n \in \mathcal{D}_{\mathbb{B}(u)}$ put $ \sigma_{(u,n)} = !_{\cod(n)}$. The data of the $(\sigma_{(u,n)})_{(u,n) \in \mathcal{D}_\mathbb{B}} $ define a natural transformation $ \sigma : \cod \Rightarrow \overline{\mathbb{B}}$ and we can exploit the universal property of the left Kan extension to deduce a natural transformation $ \overline{\mathbb{B}} \rightarrow {\iota_x}_*\widetilde{\mathbb{B}\!\mid_x}$. 
\[\begin{tikzcd}
	{\mathcal{D}_\mathbb{B}} && {\mathcal{B}} & {} \\
	{\tau_\mathbb{B}^{\op}} && {} \\
	{\tau_{\mathbb{B}\!\mid_x}^{\op}} && {}
	\arrow["{D}"', from=1-1, to=2-1]
	\arrow["{\cod}"{name=0}, from=1-1, to=1-3]
	\arrow["{\overline{\mathbb{B}}}"{name=1, description}, from=2-1, to=1-3]
	\arrow["{\iota_x^{-1}}"', from=2-1, to=3-1]
	\arrow["{\widetilde{\mathbb{B}\!\mid_x}}"{name=2, swap}, from=3-1, to=1-3, curve={height=12pt}]
	\arrow[Rightarrow, "{\zeta^{\mathbb{B}}}"', from=0, to=1, shorten <=2pt, shorten >=2pt]
	\arrow[Rightarrow, from=1, to=2, shift right=1, shorten <=3pt, shorten >=3pt, dashed]
	\arrow[Rightarrow, from=0, to=2, shift right=5, curve={height=6pt}, shorten <=4pt, shorten >=4pt, "\sigma"'{near start}, crossing over]
\end{tikzcd}\]
which factorizes through the sheafification $ \widetilde{\mathbb{B}}$. The mate of this defines a natural map $\iota_x^*\widetilde{\mathbb{B}} \rightarrow \widetilde{\mathbb{B}\!\mid_x}$, but we have at each point $\xi$ of $X_{ \mathbb{B}\!\mid_x}$ an isomorphism \[\iota_x^*\widetilde{\mathbb{B}}\mid_\xi \simeq \widetilde{\mathbb{B}}\mid_{(x,\xi)} \simeq U(A_\xi) \simeq\widetilde{\mathbb{B}\!\mid_x}\mid_\xi  \]
But as $ X_{\mathbb{B}\!\mid_x}$ has enough points, this suffices to have isomorphism of sheaves.
\end{proof}

\begin{definition}
The spectrum of a $\mathcal{B}$-space $ ((X,\tau), \mathbb{B})$ is the $ U$-space 
\[ ((X_\mathbb{B}, \tau_\mathbb{B}), \widetilde{\mathbb{B}}, (A_\xi)_{(x,\xi) \in X_\mathbb{B}})\] 
\end{definition}

Now for a morphism of $ (f, \phi) : \mathcal{B}$-space $((X_1,\tau_1), \mathbb{B}_1) \rightarrow ((X_2,\tau_2), \mathbb{B}_2)$, the inverse and direct image parts of $\phi$ are defined respectively on points $(x,\xi) \in X_{\mathbb{B}_2}$ and basic opens $ D_{(u,n)}$ of $ \tau_{\mathbb{B}_1}$ by factorization and puhsouts as follows
\[\begin{tikzcd}[sep=large]
	{\mathbb{B}_1\!\mid_{f(x)}} & {\mathbb{B}_2\!\mid_{x}} \\
	{U(A_{\xi \phi^\flat_x})} & {U(A_\xi)}
	\arrow["{n_{\xi \phi^\flat_x}}"', from=1-1, to=2-1]
	\arrow["{\xi}", from=1-2, to=2-2]
	\arrow["{\phi^\flat_x}", from=1-1, to=1-2]
	\arrow["{U(L_{A_\xi}(\xi \phi^\flat_x))}"', from=2-1, to=2-2]
\end{tikzcd} \quad \quad
\begin{tikzcd}[sep=large]
	{\mathbb{B}_1(u)} & {\mathbb{B}_2(f^{-1}(u))} \\
	{C} & {{\phi^\sharp_u}_*C}
	\arrow["{{\phi^\sharp_u}_*n}", from=1-2, to=2-2]
	\arrow["{\phi^\sharp_u}", from=1-1, to=1-2]
	\arrow["{n}"', from=1-1, to=2-1]
	\arrow["{n^*\phi^\sharp_u}"', from=2-1, to=2-2]
	\arrow["\lrcorner"{very near start, rotate=180}, from=2-2, to=1-1, phantom]
\end{tikzcd} \]

\begin{proposition}
We have a morphism of $U$-spaces
\[ \begin{tikzcd}
{((X_{\mathbb{B}_1}, \tau_{\mathbb{B}_2}),\widetilde{\mathbb{B}_1}, (A_\xi)_{(x,\xi) \in X_{\mathbb{B}_1}})} &{((X_{\mathbb{B}_2}, \tau_{\mathbb{B}_2}),\widetilde{\mathbb{B}_2}, (A_\xi)_{(x,\xi) \in X_{\mathbb{B}_2}})} 	\arrow[from=1-1, to=1-2]
\end{tikzcd} \]
whose underlying continuous map is
\[\begin{tikzcd}[row sep=tiny]
	{(X_{\mathbb{B}_2}, \tau_{\mathbb{B}_2})} & {(X_{\mathbb{B}_1}, \tau_{\mathbb{B}_2})} \\
	{(x,\xi)} & {(f(x), n_{\xi \phi^\flat_x})}
	\arrow["{\Spec(\phi)}", from=1-1, to=1-2]
	\arrow[from=2-1, to=2-2, shorten <=2pt, shorten >=2pt, maps to]
\end{tikzcd}\]
which is moreover spectral, and where the sheaf morphism $\widetilde{\phi}$ has respectively the following inverse and direct image parts at points $(x,\xi) \in X_{\mathbb{B}_2}$ and basic opens $ D_{(u,n)}$ of $ \tau_{\mathbb{B}_1}$
\[ \widetilde{\phi}^{\flat}_{(x,\xi)} = L_A(\xi \phi^\flat_x) \quad \quad \widetilde{\phi}^{\sharp}_{D_{(u,n)}} = \gamma_{\mathbb{B}_2} \zeta^{\mathbb{B}_2}_{D_{(u,n)}} n^*\phi^\sharp_u  \]
where $ \gamma_{\mathbb{B}_2} : \lan_{D_{\mathbb{B}_2}} \cod \pi_1 \Rightarrow \widetilde{\mathbb{B}_2}$ is the sheafification map and $ \zeta^{\mathbb{B}_2} : \cod \pi_1 \Rightarrow \lan_{D_{\mathbb{B}_2}} \cod \pi_1$ is the canonical natural transformation of the left Kan extension. 
\end{proposition}

\begin{proof}
For a basic open $ D_{(u,n)}$ in $\tau_{\mathbb{B}_1}$ we have $\Spec (\phi)^{-1}(D_{(u,n)}) = D_{(f^{-1}(u), {\phi^{\sharp}_u}_*n)}$ which is a basic open of $ \tau_{\mathbb{B}_2}$. 
\end{proof}

\begin{theorem}\label{generalizedadjunction}
We have an adjunction 
\[\begin{tikzcd}
	{U-Spaces} && {\mathcal{B}-Spaces}
	\arrow["{\iota_u}"{name=0, swap}, from=1-1, to=1-3, curve={height=20pt}]
	\arrow["{\Spec}"{name=1, swap}, from=1-3, to=1-1, curve={height=20pt}]
	\arrow["\dashv"{rotate=-90}, from=1, to=0, phantom]
\end{tikzcd}\]
\end{theorem}

\begin{proof}
Suppose we have a morphism of $\mathbb{B}$-spaces $ (f,\phi) : ((X,\tau), \mathbb{B}) \rightarrow \iota_U((Y,\sigma),\mathbb{A}, (A_y)_{y \in Y}) $. Then in each $ y \in Y$ we have a factorization
\[\begin{tikzcd}
	{\mathbb{B}\!\mid_{f(y)}} & {} & {U(A_y)} \\
	& {U(A_{\phi^\flat_y})}
	\arrow["{\phi^\flat_y}", from=1-1, to=1-3]
	\arrow["{\eta^{A_y}_{\phi^\flat_y}}"', from=1-1, to=2-2]
	\arrow["{U(L_{A_y}({\phi^\flat_y}))}"', from=2-2, to=1-3]
\end{tikzcd}\]
producing a local form under the stalk $ \mathbb{B}\!\mid_{f(y)}$ so we can define 
\[\begin{tikzcd}[row sep=tiny]
	{Y} & {X_\mathbb{B}} \\
	{y} & {(f(y), \eta^{A_y}_{\phi^\flat_y})}
	\arrow["{g}", from=1-1, to=1-2]
	\arrow[from=2-1, to=2-2, shorten <=2pt, shorten >=2pt, maps to]
\end{tikzcd}\]
From this point, the strategy is similar to the set valued case, though we have to work under the values of the structural sheaf under an open rather than just considering a map produced by global section. We prove that $ g $ is continuous. Let be $ u \in \tau$ and $ n \in \mathcal{D}_{\mathbb{B}(u)}$. We have that $g^{-1}(D_{(u,n)})$ consists of all those $y \in Y$ such that we have a factorization
\[\begin{tikzcd}
	{\mathbb{B}(u)} & {C} \\
	{\mathbb{B}\!\mid_{f(y)}} & {U(A_{\phi^\flat_y})}
	\arrow["{\eta^{A_y}_{\phi^\flat_y}}"', from=2-1, to=2-2]
	\arrow["{\rho^{u}_y}"', from=1-1, to=2-1]
	\arrow[from=1-2, to=2-2]
	\arrow["{n}", from=1-1, to=1-2]
\end{tikzcd}\]
Now for each such $y \in g^{-1}(D_{(u,n)})$ we have a factorization as the dashed arrow below
\[\begin{tikzcd}
	{\mathbb{B}(u)} && {\mathbb{A}(f^{-1}(u))} \\
	{C} && {{\phi^\sharp_u}_*C} \\
	& {U(A_{\phi^\flat_y})} && {U(A_y)}
	\arrow["{\phi^\sharp_u}", from=1-1, to=1-3]
	\arrow["{{\phi^\sharp_u}_*n}", from=1-3, to=2-3]
	\arrow["{n}"', from=1-1, to=2-1]
	\arrow["{n^*\phi^\sharp_u}" {near end} description, from=2-1, to=2-3]
	\arrow[from=2-1, to=3-2]
	\arrow["{s}"', from=2-3, to=3-4, dashed]
	\arrow[from=3-2, to=3-4, "U(L_{A_y}(\phi^\flat_y))"']
	\arrow["{\eta^{A_y}_{\phi^\flat_y}\rho^u_{f(y)}}" description, from=1-1, to=3-2, curve={height=-12pt}, crossing over]
	\arrow["{\rho^{f^{-1}(u)}_y}", from=1-3, to=3-4, curve={height=-12pt}]
\end{tikzcd}\]
But we have a filtered colimit \[ \rho^{f^{-1}}_y = \underset{y \in v \atop v \subseteq f^{-1}(u)}{\colim} \, \rho^{f^{-1}(u)}_v \] in the slice $ \mathbb{A}(f^{-1}(u))\downarrow \mathcal{B}$ for $\mathbb{A}$ is a sheaf and the set $ \{ v \in \sigma \mid \, v \subseteq f^{-1}(u), \, y \in v \}$ is cofinal in the set of neighborhood of $y$ in $\sigma$. Hence, for $ {\phi^\sharp_u}_*n$ is still finitely presented as a diagonally universal map, there is some neighborhood $v$ of $y$ in $\sigma$ such that $ v \subseteq f^{-1}(u)$ and the arrow $s$ factorizes through the restriction $ \rho^{v}_y$, and moreover, any two such factorizations can be equalized by a third one. But now any point $z$ in $v$ is also in $g^{-1}(D_{(u,n)})$: indeed, if we have a factorization as below
\[\begin{tikzcd}[row sep=large]
	{\mathbb{B}(u)} && {\mathbb{A}(f^{-1}(u))} \\
	{C} & {} & {{\phi^\sharp_u}_*C} & {\mathbb{A}(v)} \\
	& {\mathbb{B}\!\mid_{f(z)}} && {U(A_z)}
	\arrow["{\phi^\flat_z}", from=3-2, to=3-4]
	\arrow["{\phi^\sharp_u}", from=1-1, to=1-3]
	\arrow["{n}"', from=1-1, to=2-1]
	\arrow["{{\phi^\sharp_u}_*n}", from=1-3, to=2-3]
	\arrow[from=2-1, to=2-3]
	\arrow["{\rho^{f^{-1}(u)}_v}", from=1-3, to=2-4, curve={height=-12pt}]
	\arrow[from=2-3, to=2-4]
	\arrow["{\rho^v_z}", from=2-4, to=3-4]
	\arrow["{\rho^u_{f(z)}}"{pos=0.3, description}, from=1-1, to=3-2, curve={height=-12pt}, crossing over]
\end{tikzcd}\]
then ${\rho^u_{f(z)}}_*n$ factorizes $\phi^\flat_z$, and hence by Diers condition of the local unit of a morphism toward $U$ being the filtered colimit of all finitely presented diagonally universal maps factorizing it, we have that ${\rho^u_{f(z)}}_*n \leq \eta^{A_z}_{\phi^\flat_z}$, so that $z \in g^{-1}(D_{(u,n)}$ for $f(z) \in u$ and the last condition. Hence $v$ is included in $g^{-1}(D_{(u,n)})$ and as before such open $v$ could have be chosen for each point of $g^{-1}(D_{(u,n)})$, which is hence open: this ensures the continuity of $g$. Moreover observe that $ g^{-1}(D_{(u,n)}) \subseteq f^{-1}(u) = g^{-1}(D_{(u,1_{\mathbb{B}(u)})})$.  \\

The remains of the proof is similar to the proof of the point-based case: the opens $ v \in \sigma$ constructed above defined a covering of $g^{-1}(D_{(u,n)})$, and the universal property of the limit from the sheaf condition of $ \mathcal{A}$ at this cover provides a map as the dashed arrow below
\[\begin{tikzcd}
	{\mathbb{B}(u)} & {\mathbb{A}(f^{-1}(u))} \\
	{C} & {\mathbb{A}(g^{-1}(D_ {(u,n)}))}
	\arrow["{\rho^{f^{-1}(u)}_{g^{-1}(D_ {(u,n)})}}", from=1-2, to=2-2]
	\arrow["{\phi^\sharp_u}", from=1-1, to=1-2]
	\arrow["{n}"', from=1-1, to=2-1]
	\arrow["{\sigma_{(u,n)}}"', from=2-1, to=2-2, dashed]
\end{tikzcd}\]

Again, the data of all the $(\sigma_{(u,n)})_{u \in \tau, \, n \in \mathcal{D}_{\mathbb{B}(u)}}$ provide a natural transformation $ \sigma : \cod \pi_2 \Rightarrow g_*\mathbb{A}$ and we apply the property of left Kan extension
\[\begin{tikzcd}[column sep=large]
	{\mathcal{D}_\mathbb{B}} & {} & {\mathcal{B}} \\
	{\tau_{\mathbb{B}}^{op}} & {} \\
	{\sigma}
	\arrow["{D}"', from=1-1, to=2-1]
	\arrow["{g^{-1}}"', from=2-1, to=3-1]
	\arrow["{\cod \pi_2}"{name=0}, from=1-1, to=1-3]
	\arrow["{\lan_D\cod \pi_2}"{name=1, description}, from=2-1, to=1-3]
	\arrow["{\mathbb{A}}"{name=2, swap}, from=3-1, to=1-3, curve={height=12pt}]
	\arrow[Rightarrow, "{\zeta_\mathbb{B}}"', from=0, to=1, shorten <=2pt, shorten >=2pt]
	\arrow[Rightarrow, "{\sigma}"'{near start}, from=0, to=2, shift right=10, curve={height=6pt}, shorten <=4pt, shorten >=4pt, crossing over]
	\arrow[Rightarrow, "{\overline{\psi}}", from=1, to=2, shorten <=3pt, shorten >=3pt, dashed]
\end{tikzcd}\]
and again this $ \overline{\psi}$ factorizes through the sheafification $ \widetilde{\mathbb{B}} \rightarrow g_*\mathbb{A}$.  \\

Now concerning the inverse image part, take the mate of $ \psi^\sharp$ and apply again Diers condition and \cref{stalk lemma}. For any $ y \in Y$ and $(u,n) \in \mathcal{D}_{\mathbb{B}} $ with $ g(y) \in D_{(u,n)}$, the following square 
\[\begin{tikzcd}
	{\widetilde{\mathbb{B}}(D_{(u,n)}))} & {\mathbb{A}(g^{-1}(D_{(u,n)})} \\
	{\widetilde{\mathbb{B}}\!\mid_{g(y)}} & {U(A_y)}
	\arrow["{\rho^{g^{-1}(D_{(u,n)})}_{y}}", from=1-2, to=2-2]
	\arrow["{\psi^\sharp_{D_{(u,n)}}}", from=1-1, to=1-2]
	\arrow["{\rho^{D_{(u,n)}}_{g(y)}}"', from=1-1, to=2-1]
	\arrow["{\psi^\flat_y}"', from=2-1, to=2-2]
\end{tikzcd}\]
is part of the following diagram expressing the factorization through the unit
\[\begin{tikzcd}
	{\mathbb{B}(u)} & {\widetilde{\mathbb{B}}(D_{(u , 1_{\mathbb{B}(u)})})} & {\mathbb{A}(f^{-1}(u))} \\
	{} & {\widetilde{\mathbb{B}}(D_{(u,n)})} & {\mathbb{A}(g^{-1}(D_{(u,n)})} \\
	{\mathbb{B}\!\mid_{f(y)}} & {} & {} \\
	& {\widetilde{\mathbb{B}}\!\mid_{g(y)}} & {U(A_y)}
	\arrow["{\psi^\sharp_{D_{(u,n)}}}", from=2-2, to=2-3]
	\arrow["{\eta^\sharp_u}", from=1-1, to=1-2]
	\arrow["{\rho^{D_{(u , 1_{\mathbb{B}(u)})}}_{D_{(u,n)}}}"', from=1-2, to=2-2]
	\arrow["{\rho^{f^{-1}(u)}_{g^{-1}(D_{(u,n)}}}", from=1-3, to=2-3]
	\arrow["{\psi^\sharp_{D_{(u,1_{\mathbb{B}(u)})}}}", from=1-2, to=1-3]
	\arrow["{\eta^\flat_{f(y)}}"', from=3-1, to=4-2]
	\arrow["{\psi^\flat_y}"', from=4-2, to=4-3]
	\arrow["{\rho^{g^{-1}(D_{(u,n)})}_{y}}", from=2-3, to=4-3]
	\arrow["{\rho^{u}_{g(y)}}"', from=1-1, to=3-1]
	\arrow["{\phi^\flat_y}"{near end}, from=3-1, to=4-3, curve={height=-20pt}]
	\arrow["{\rho^{D_{(u,n)}}_{g(y)}}" description, from=2-2, to=4-2, crossing over]
	\arrow["{\phi^\sharp_u}", from=1-1, to=1-3, curve={height=-30pt}]
\end{tikzcd}\]
where the arrow $\psi^\flat_y$ is induced by the universal property of the filtered colimit 
\[ \widetilde{\mathbb{B}}\!\mid_{g(y)} = \underset{n \leq \eta^{A_y}_{\phi^\flat_y}}{\colim} \, \widetilde{ \mathbb{B}\! \mid_{f(y)}}(D_n) \simeq U(A_{ \phi^\flat_y}) \]
as the map $ U(L_{A_y}(\phi^\flat_y))$. 
\end{proof}

To conclude this part, let us explain how this adjunction is related to the adjunction existing at the level of the free coproduct completion as exposed in \cite{partI}. Recall that a right multi-adjoint $U : \mathcal{A} \rightarrow \mathcal{B}$ induced an adjunction
\[\begin{tikzcd}
	{\Pi\mathcal{A}} && {\Pi\mathcal{B}}
	\arrow["{\Pi U}"{name=0, swap}, from=1-1, to=1-3, curve={height=15pt}]
	\arrow["{L}"{name=1, swap}, from=1-3, to=1-1, curve={height=15pt}]
	\arrow["\dashv"{rotate=-90}, from=1, to=0, phantom]
\end{tikzcd}\]
Now we can define a ``functor of stalks" $ \pi_\mathcal{B} : \mathcal{B}-Spaces \rightarrow \Pi \mathcal{B} $ sending a $\mathcal{B}$-space $ ((X,\tau),\mathbb{B})$ on the family of stalks
\[\begin{tikzcd}
	{X} & {\mathcal{B}}
	\arrow["{(\mathbb{B}\mid_x)_{x \in X}}", from=1-1, to=1-2]
\end{tikzcd}\] and a morphism $ (f,\phi) : ((X_1, \tau_1),\mathbb{B}_1) \rightarrow ((X_2, \tau_2),\mathbb{B}_2)$ to 
\[\begin{tikzcd}
	{X_1} & {} \\
	&& {\mathcal{B}} \\
	{X_2}
	\arrow["{(\mathbb{B}_1\!\mid_x)_{x \in X_1}}"{name=0}, from=1-1, to=2-3]
	\arrow["{f}", from=3-1, to=1-1]
	\arrow["{(\mathbb{B}_2\!\mid_x)_{x \in X_2}}"{name=1, swap}, from=3-1, to=2-3]
	\arrow[Rightarrow, "{(\phi^\flat_x)_{x \in X_2}}"', from=0, to=1, shorten <=5pt, shorten >=5pt, shift right=7]
\end{tikzcd}\]   
Similarly, we can define a functor $ \pi_\mathcal{A} : U-Spaces \rightarrow \mathcal{A}$ sending a $U$-space on the specified family of $\mathcal{A}$-objects attached to its stalks, and a morphism of $U$-space to the specified maps in $\mathcal{A}$ attached to the inverse image part at the stalks. Then the following result is obvious from the way we defined the left adjoint $L \dashv \Pi U$ in \cite{partI}:

\begin{proposition}\label{Pro is discrete}
The following square satisfies the Beck-Chevalley condition
\[\begin{tikzcd}[row sep=large]
	{U-Spaces} & {\mathcal{B}-Spaces} \\
	{\Pi\mathcal{A}} & {\Pi\mathcal{B}}
	\arrow["{\pi_\mathcal{A}}"', from=1-1, to=2-1]
	\arrow["{\pi_\mathcal{B}}", from=1-2, to=2-2]
	\arrow["{L}"{name=0, swap}, from=2-2, to=2-1, curve={height=12pt}]
	\arrow["{\Spec}"{name=1, swap}, from=1-2, to=1-1, curve={height=12pt}]
	\arrow["{\Pi U}"{name=2, swap}, from=2-1, to=2-2, curve={height=12pt}]
	\arrow["{\iota_U}"{name=3, swap}, from=1-1, to=1-2, curve={height=12pt}]
	\arrow["\dashv"{rotate=-90}, from=0, to=2, phantom]
	\arrow["\dashv"{rotate=-90}, from=1, to=3, phantom]
\end{tikzcd}\]
that is we have both that $ \Pi U \pi_\mathcal{A} = \pi_\mathcal{B} \iota_U$ and $ \pi_\mathcal{A} \Spec = L \pi_\mathcal{B}$
\end{proposition}

In some sense, the functors $ \pi_\mathcal{B}$ and $\pi_\mathcal{A}$ forget about the topological data attached to structured spaces and remind only the local data at stalks. But in a converse process, we could see families in the free product completion as \emph{discrete} structured spaces: indeed we can define two functors \[ \iota_\mathcal{B} : \Pi \mathcal{B} \hookrightarrow \mathcal{B}-Spaces\] sending a family $(B_i)_{i \in I}$ to the $\mathcal{B}$-space $ ((I, \mathcal{P}(I)), \mathbb{B})$ where the set $I$ is now equiped with the discrete topology and the sheaf $\mathbb{B}$ is defined as $ \mathbb{B}(J) = \prod_{i \in J}B_i$ for $ J \subseteq I$. Similarly we define a functor \[\iota_\mathcal{A} : \Pi \mathcal{A} \hookrightarrow U-Spaces \] sending a family $ (A_i)_{i \in I}$ to the discrete $U$-space $((I, \mathcal{P}(I)), \mathbb{A}, (A_i)_{i \in I})$ where $ \mathbb{A}(J) = \prod_{i \in I} U(A_i)$ for $J \subseteq I$.

\section{2-Functoriality of Diers construction and 2-category of dualities}

In this section, we define a convenient 2-categorical context for Diers construction, introducing both a 2-category of ``Diers contexts" and a 2-category of ``geometric dualities" abstracting the spectral adjunctions we get from the construction. First we introduce an adequate notion of morphisms of Diers contexts and explain how they are transported along Diers construction and its subsequent generalization as developed in section 2. Such morphisms will later reveal to allow transfer of geometric information from a duality to another one. \\

\begin{definition}\label{diers contexts}
We define the 2-category ${\mathfrak{Diers}}$ of \emph{Diers contexts} as having\begin{itemize}
    \item as 0-cells, triples $(U, \mathcal{A}, \mathcal{B})$ locally right adjoint functors $ U : \mathcal{A} \rightarrow \mathcal{B} $ with $\mathbb{B}$ locally finitely presentable and satisfying Diers condition
    \item as 1-cells $ (U_1, \mathcal{A}_1, \mathcal{B}_1) \rightarrow (U_2, \mathcal{A}_2, \mathcal{B}_2)$, triples of the form $(F,G, \theta) $ with $ F : \mathcal{A}_1 \rightarrow \mathcal{A}_2$ a functor and $ G : \mathcal{B}_1 \rightarrow \mathcal{B}_2$ a morphism of locally finitely presentable categories, and $\theta$ an invertible 2-cell (which we will leave implicit in general)
\[ \begin{tikzcd}
\mathcal{A}_1 \arrow[]{r}{U_1} \arrow[]{d}[swap]{F}\arrow[phantom]{rd}{\theta \atop \simeq} & \mathcal{B}_1 \arrow[]{d}{G}  \\
\mathcal{A}_2 \arrow[]{r}[swap]{U_2} & \mathcal{B}_2
\end{tikzcd}
\]
\item and, as 2-cells $ (F,G,\theta) \Rightarrow (F',G',\theta')$, pairs of natural transformations $ (\alpha, \beta)$  
\[ \begin{tikzcd}[sep =large]
\mathcal{A}_1 \arrow[]{r}{U_1} \arrow[bend right=30, ""{name=D, inner sep=0.5pt}]{d}[swap]{F} \arrow[bend left=30, ""{name=U, inner sep=0.5pt}]{d}{F'} \arrow[Rightarrow, from=D, to=U]{}{\alpha} & \mathcal{B}_1 \arrow[bend right=30, ""{name=D', inner sep=0.5pt}]{d}[swap]{G} \arrow[bend left=30, ""{name=U', inner sep=0.5pt}]{d}{G'} \arrow[Rightarrow, from=D', to=U']{}{\beta} \\
\mathcal{A}_2 \arrow[]{r}{U_2} & \mathcal{B}_2
\end{tikzcd} \]
such that the two whiskering coincide through the provided iso : $ \theta' U_{2*} \alpha = U_1^*\beta \theta$.
\end{itemize} 
\end{definition}  
\begin{remark}
First, a remark about 1-cells: we recall that a morphism of locally finitely presentable categories is a functor $ G : \mathcal{B}_1 \rightarrow \mathcal{B}_2$ preserving filtered colimits and admitting a left adjoint $G^* :\mathcal{B}_2 \rightarrow \mathcal{B}_1 $. It can be shown that this left adjoint sends finitely presented objects of $\mathcal{B}_2$ to finitely presented object of $\mathcal{B}_1$ - this is equivalent to preservation of filtered colimits by $G$. \\

 Moreover, observe that, for a morphism $ (F,G,\theta)$, we have in each $A \in \mathcal{A}_1$ a pseudo commutative square
\[\begin{tikzcd}
	{\mathcal{A}_1/_A} & {\mathcal{B}_1/_{U_1(A)}} \\
	{\mathcal{A}_2/_{F(A)}} & {\mathcal{B}_2/_{U_2F(A)}}
	\arrow["{U_1/_A}", from=1-1, to=1-2]
	\arrow["{F/_A}"', from=1-1, to=2-1]
	\arrow["{G/_{U_1(A)}}", from=1-2, to=2-2]
	\arrow["{U_2/_{F(A)}}"', from=2-1, to=2-2]
	\arrow["{\theta_A \atop \simeq}" description, from=1-1, to=2-2, phantom, no head]
\end{tikzcd}\]
which is equiped with a canonical mate $ \sigma^A : L^2_{F(A)} \circ G \Rightarrow F \circ L^1_A $ as depicted below 
\[\begin{tikzcd}
	{\mathcal{A}_1/_A} & {\mathcal{B}_1/_{U_1(A)}} \\
	{\mathcal{A}_2/_{F(A)}} & {\mathcal{B}_2/_{U_2F(A)}}
	\arrow["{F/_A}"{name=0, swap}, from=1-1, to=2-1]
	\arrow["{G/_{U_1(A)}}", from=1-2, to=2-2]
	\arrow["{L^1_A}"', from=1-2, to=1-1]
	\arrow["{L^2_{F(A)}}"{name=1}, from=2-2, to=2-1]
	\arrow[Rightarrow, "{\sigma_A}"', from=1, to=0, curve={height=12pt}, shorten <=4pt, shorten >=4pt]
\end{tikzcd}\]
As we shall see, this mate acts as a comparison map between the $U_2$-factorization of the image by $G$ and the image by $G$ of the $U_1 $-factorization. \\

Even though we shall not have major use of it, it is worth giving some precision on 2-cells, especially their interaction with the canonical mate. First observe that on the slices $\alpha$ induces a transformation
\[\begin{tikzcd}
& \mathcal{A}_1/A \arrow[]{dl}[swap]{F/A} \arrow[""{name=U, inner sep=0.5pt, below}]{dr}{F'/A} & \\
\mathcal{A}_2/F(A) \arrow[""{name=D, inner sep=0.5pt}]{rr}[swap]{{\alpha_A}_!} & & \mathcal{A}_2/F'(A) \arrow[Rightarrow, bend left=20, from=D, to=U]{}{\alpha/A}
\end{tikzcd}\]
where ${\alpha_A}_! $ is the postcomposition with the component of $\alpha$ in $A $, $ \alpha_A : F(A) \rightarrow F'(A)$, and in any  $u \in \mathcal{A}_1/A$, $\alpha/A$ has  for component $(\alpha/A)_u  $, which is indeed an arrow in $\mathcal{A}_2/F'(A)$ from the commutativity of the naturality square of $\alpha$. We have the same for $ \beta$ and this induces the following diagram
\[
\begin{tikzcd}[row sep=large,column sep=large]
& \mathcal{A}_1/A \arrow[]{rr}{U_1/A} \arrow[]{ld}{F/A} \arrow[""{name=U, inner sep=0.5pt, near end}, near start]{dd}{F'(A)} & & \mathcal{B}_1/U_1(A) \arrow[]{ld}{G/U_1(A)} \arrow[""{name=U', inner sep=0.5pt, near end}]{dd}{G'/U_1(A)} \\
\mathcal{A}_2/F(A) \arrow[""{name=D, inner sep=0.5pt}]{rd}[swap]{{\alpha_A}_!} \arrow[dash]{r}{}& \arrow[]{r}{U_2/F(A)}& \mathcal{B}_2/GU_1(A) \arrow[""{name=D', inner sep=0.5}]{rd}[swap]{\beta_{U_1(A)!}} & \\
& \mathcal{A}_2/F'(A) \arrow[]{rr}{U_2/F'(A)} && \mathcal{B}_2/G'U_1(A) \arrow[Rightarrow, from=D, to=U, bend left=20]{}{\alpha/_A} \arrow[Rightarrow, from=D', to=U', bend left=20]{}{\beta/_{U_1(A)}}
\end{tikzcd}
\]
where the lower square commutes from $ GU_1(A) = U_2F(A)$ (resp $G'U_1(A) = U_2F'(A)$) and the equality of whiskering. In fact this square does even satisfy Beck-Chevalley condition, that is the following square also commutes :
\[
\begin{tikzcd}
\mathcal{A}_2/F(A) \arrow[]{d}[swap]{{\alpha_A}_!}
&\arrow[]{l}[swap]{L^2_{F(A)}} \mathcal{B}_2/GU_1(A) \arrow[]{d}{\beta_{U_1(A)!}} \\
\mathcal{A}_2/F'(A)  & \arrow[]{l}{L^2_{F'(A)}}  \mathcal{B}/G'U_1(A)
\end{tikzcd}
 \]
This comes from the fact that post composing with an arrow in the range of a stable functor does not modify the candidate part of the factorization (up to isomorphism). Hence for any arrow $f : B \rightarrow U_2F(A)$ in $\mathcal{B}_2$, we have a canonical isomorphism between the factorizations
\[\begin{tikzcd}
B \arrow[]{rd}[swap]{\eta^{F(A)}_f} \arrow[]{rr}{x} && GU_1(A)\simeq U_2F(A)  \arrow[]{r}{U_2(\alpha_A)} & G'U_1(A) \simeq U_2F'(A) \\
& U_2L^2_{F(A)}(A_f)  \arrow[]{ru}{U_2(u_x)}  \arrow[]{r}{\simeq} & U_2L^2_{F'(A)}(U_2(\alpha_A)x) \arrow[]{ru}[swap]{U_2(u_{U_2(\alpha_A)x})}
\end{tikzcd}\]
This provides an equality between the pasting of the sliced transformations with the respective mates :
\[ \alpha/_A*\sigma_A = \sigma'_A*\beta/_{U_1(A)} \]
making the following diagram to commute through a 2-dimensional equality
\[
\begin{tikzcd}[row sep=large,column sep=large]
& \mathcal{A}_1/A  \arrow[""{name=A, inner sep=0.5pt}]{ld}[swap]{F/A} \arrow[""{name=U, inner sep=0.5pt, near end}, ""{name=B', inner sep=0.5pt, near end}, near start]{dd}{F'(A)} & & \arrow[]{ll}[swap]{L^1_A} \mathcal{B}_1/U_1(A) \arrow[]{ld}{G/U_1(A)} \arrow[""{name=U', inner sep=0.5pt, near end}]{dd}{G'/U_1(A)} \\
\mathcal{A}_2/F(A) \arrow[""{name=D, inner sep=0.5pt}]{rd}[swap]{{\alpha_A}_!} & \arrow[""{name=B, inner sep=0.5pt}]{l}{} \arrow[Rightarrow, bend right=20, from=B, to=A]{}{\sigma_A} & \arrow[dash]{l}{L^2_{F(A)}} \mathcal{B}_2/GU_1(A) \arrow[""{name=D', inner sep=0.5}]{rd}[swap]{\beta_{U_1(A)!}} & \\
& \mathcal{A}_2/F'(A)  && \arrow[""{name=A', inner sep=0.5pt, near end}]{ll}{L^2_{F'(A)}} \arrow[Rightarrow, bend right=20, from=A', to=B']{}{\sigma'_A}' \mathcal{B}_2/G'U_1(A) \arrow[Rightarrow, from=D, to=U, bend left=20]{}{\alpha/_A} \arrow[Rightarrow, from=D', to=U', bend left=20]{}{\beta/_{U_1(A)}}
\end{tikzcd}
\]
\end{remark}

Now we must provide an abstraction of the adjunctions one gets from Diers construction and its generalization to arbitrary structured spaces. \\

Now to guess what should be at the other end of our construction, let us give some observation on the nature of the categories of $U$-spaces and $\mathbb{B}$-spaces. \\

First, remark that $\mathcal{B}-Spaces$ is the opposite category of the Grothendieck construction associated to the pseudofunctor   \[ \begin{array}{rcl}
        \mathcal{T}op^{op} & \rightarrow & \mathcal{C}at \\
         (X, \tau) & \mapsto & \Sh_\mathcal{B}(X, \tau) \\ 
         (X_1, \tau_1) \stackrel{f}{\rightarrow} (X_2,\tau_2) & \mapsto &          \Sh_\mathcal{B}(X_2,\tau_2) \stackrel{f^*}{\rightarrow} {S}h_\mathcal{B}(X_1, \tau_1) 
    
    \end{array} \]
and hence has the structure of a fibration: that is for each continuous map $f : (X_1, \tau_1) \rightarrow (X_2, \tau_2)$, we have a cartesian lift $ ((X_2, \tau_2), \mathbb{B}) \rightarrow ((X_1, \tau_1), f^*\mathbb{B}) $, as for any situation as below
\[\begin{tikzcd}[row sep=small]
	& {((X_1, \tau_1), f^*\mathbb{B})} & {((X_2, \tau_2), \mathbb{B})} \\
	{((X_3, \tau_3), \mathbb{B}')} \\
	& {(X_1, \tau_1)} & {(X_2, \tau_2)} \\
	{(X_3, \tau_3)}
	\arrow["{g}", from=4-1, to=3-2]
	\arrow["{f}", from=3-2, to=3-3]
	\arrow["{h}"', from=4-1, to=3-3, curve={height=12pt}]
	\arrow["{(f, 1_{f^*\mathbb{B}})}"', from=1-3, to=1-2]
	\arrow["{(h, \phi)}", from=1-3, to=2-1, curve={height=-12pt}]
\end{tikzcd}\]
we have $ h^*= g^*f^*$ so that $ \phi : h^*\mathbb{B} = g^*f^*\mathbb{B} \rightarrow \mathbb{B}' $ provides itself automatically a lift $ (g, \phi ) : {((X_3, \tau_3), \mathbb{B}')}  \rightarrow  {((X_1, \tau_1), f^*\mathbb{B})} $. Moreover this structure of fibration restricts to $ U-Spaces$ as inverse image preserves stalks, and the cartesian lifts are identities between sheaves, hence have identities at stalks, which are in the range of $U$. Hence we have a morphism of fibrations
\[\begin{tikzcd}[row sep=small]
	{U-Spaces} && {\mathcal{B}-Spaces} \\
	& {\mathcal{T}op^{op}}
	\arrow["{\iota_U}", from=1-1, to=1-3]
	\arrow[from=1-1, to=2-2]
	\arrow[from=1-3, to=2-2]
\end{tikzcd}\]
for the forgetful functor $\iota_U$ does not modify the underlying topological space. Observe that $ \mathbb{B}-Space$ has also a structure of opfibration thanks to direct image, but this structure is not preserved by the restriction of $U$-spaces\\

\begin{definition}\label{dualities}
We define the 2-category $ \mathfrak{Dual}$ of \emph{spectral dualities} as having:\begin{itemize}
    \item as 0-cells, quadruples $(\mathfrak{A}, \mathfrak{B}, \iota, \mathfrak{S})$ with $\mathfrak{A}$ and $\mathfrak{B}$ fibered categories over $ \mathcal{T}op$ such that moreover the fibration of $ \mathfrak{B}$ is also an opfibration over $\mathbb{T}op$, $\iota$ a strict morphism of fibrations 
\[\begin{tikzcd}
	{\mathfrak{A}} & {} & \mathfrak{B} \\
	& {\mathcal{T}op^{op}}
	\arrow["{\iota}", from=1-1, to=1-3]
	\arrow["{p}"', from=1-1, to=2-2]
	\arrow[from=1-3, to=2-2]
\end{tikzcd}\]
and an adjunction 
\[\begin{tikzcd}
	{\mathfrak{A}} && {\mathfrak{B}}
	\arrow["{\iota}"{name=0, swap}, from=1-1, to=1-3, curve={height=12pt}]
	\arrow["{\mathfrak{S}}"{name=1, swap}, from=1-3, to=1-1, curve={height=12pt}]
	\arrow["\dashv"{rotate=-90}, from=1, to=0, phantom]
\end{tikzcd}\]
\item as 1-cells $(\mathfrak{A}_1, \mathfrak{B}_1, \iota_1, \mathfrak{S}_1) \rightarrow (\mathfrak{A}_2, \mathfrak{B}_2, \iota_2, \mathfrak{S}_2)$, triples $ (\mathbb{F}, \mathbb{G}, \theta)$ forming a squares of morphisms of fibrations 
\[\begin{tikzcd}
	{\mathfrak{A}_1} && {\mathfrak{B}_1} \\
	& {\mathcal{T}op^{op}} \\
	{\mathfrak{A}_2} && {\mathfrak{B}_2}
	\arrow["{\iota_1}", from=1-1, to=1-3]
	\arrow["{\iota_2}"', from=3-1, to=3-3]
	\arrow["{\mathbb{F}}"', from=1-1, to=3-1]
	\arrow["{\mathbb{G}}", from=1-3, to=3-3]
	\arrow[from=1-1, to=2-2]
	\arrow[from=1-3, to=2-2]
	\arrow[from=3-1, to=2-2]
	\arrow[from=3-3, to=2-2]
\end{tikzcd}\]
where all triangles commute strictly and $\theta$ is an invertible 2-cell $ \mathbb{G}\iota_1 \simeq \iota_2\mathbb{F}$. \item as 2-cells $ (\alpha, \beta) : (\mathbb{F}_1, \mathbb{G}_1, \theta_1) \Rightarrow (\mathbb{F}_2, \mathbb{G}_2, \theta_2)$ such that we have equality of whiskering 
\[\begin{tikzcd}
	{{\mathfrak{A}_1} } && {{\mathfrak{B}_1} } \\
	\\
	{{\mathfrak{A}_2} } && {{\mathfrak{B}_2} }
	\arrow["{\mathbb{F}_1}"{name=0, swap}, from=1-1, to=3-1, curve={height=12pt}]
	\arrow["{\iota_1}", from=1-1, to=1-3]
	\arrow["{\mathbb{G}_1}"{name=1, swap}, from=1-3, to=3-3, curve={height=12pt}]
	\arrow["{\iota_2}"', from=3-1, to=3-3]
	\arrow["{\mathbb{F}_2}"{name=2}, from=1-1, to=3-1, curve={height=-12pt}]
	\arrow["{\mathbb{G}_2}"{name=3}, from=1-3, to=3-3, curve={height=-12pt}]
	\arrow[Rightarrow, "{\alpha}", from=0, to=2, shorten <=4pt, shorten >=4pt]
	\arrow[Rightarrow, "{\beta}", from=1, to=3, shorten <=4pt, shorten >=4pt]
\end{tikzcd}\]
\end{itemize}
\end{definition}

\begin{remark}
In the following, for an object $\mathbb{B}$ in $\mathfrak{B}$ we denote as $ |\mathbb{B}|$ for the underlying object (similarly, $|f|$ for an arrow), and this notation is transfered to $\mathfrak{A}$ as $|\mathbb{A}| = |\iota\mathbb{A}|$. Beware that in our condition we require that this fibration $|-|$ be itself also an opfibration. \\

Observe that any morphism of duality $(\mathbb{F}, \mathbb{G}, \theta)$ admits a mate 
\[\begin{tikzcd}
	{\mathfrak{A}_1} & {\mathfrak{B}_1} \\
	{\mathfrak{A}_2} & {\mathfrak{B}_2}
	\arrow["{\mathfrak{S}_1}"', from=1-2, to=1-1]
	\arrow["{\mathfrak{S}_2}"{name=0}, from=2-2, to=2-1]
	\arrow["{\mathbb{F}}"{name=1, swap}, from=1-1, to=2-1]
	\arrow["{\mathbb{G}}", from=1-2, to=2-2]
	\arrow[Rightarrow, "{\sigma}"', from=0, to=1, curve={height=6pt}, shorten <=3pt, shorten >=3pt]
\end{tikzcd}\]
In the following we are going to give some interest to the information carried by this mate. Moreover, as well as we can recover the classical $\Gamma \dashv \Spec$ adjunction by looking at the fiber over $*$, we can exploit the condition that $ \mathfrak{B}$ is also an opfibered category to compute a ``global section functor". For any $ \mathbb{A}$ in $\mathfrak{A}$, with projection $ \mid \mathbb{A} \mid$ in $\mathcal{T}op$, we have the opcartesian lift in $ \mathfrak{B}$
\[\begin{tikzcd}[row sep=tiny]
	{\iota(\mathbb{A})} & {{!_{\mid \mathbb{A} \mid}}_*\iota(\mathbb{A})} \\
	{|\mathbb{A}|} & {*}
	\arrow["{!}", from=2-1, to=2-2]
	\arrow["{\overline{{!_{\mid \mathbb{A} \mid}}}}"', from=1-2, to=1-1]
\end{tikzcd}\]
and for any morphism $ f: \mathbb{A}_1 \rightarrow \mathbb{A}_2$ we end up with two distinct lifts over the point related by the following arrow 
\[\begin{tikzcd}[row sep=small]
	& {\iota(\mathbb{A}_1)} & {{!_{\mid \mathbb{A}_1 \mid}}_*\iota(\mathbb{A}_1)} \\
	{\iota(\mathbb{A}_2)} && {{!_{\mid \mathbb{A}_2 \mid}}_*\iota(\mathbb{A}_2) \simeq {!_{\mid \mathbb{A}_1 \mid}}_*|f|_*\iota(\mathbb{A}_2)} \\
	& {|\mathbb{A}_1|} & {*} \\
	{|\mathbb{A}_2|}
	\arrow["{!_{| \mathbb{A}_1|}}", from=3-2, to=3-3]
	\arrow["{\overline{{!_{\mid \mathbb{A}_1 \mid}}}}"', from=1-3, to=1-2]
	\arrow["{|f|}", from=4-1, to=3-2]
	\arrow["{!_{| \mathbb{A}_2|}}"', from=4-1, to=3-3, curve={height=12pt}]
	\arrow["{\iota(f)}"', from=1-2, to=2-1]
	\arrow["{\overline{{!_{\mid \mathbb{A}_2 \mid}}}}", from=2-3, to=2-1]
	\arrow["{{!_{\mid \mathbb{A}_1 \mid}}_*(\iota(f))}"', from=1-3, to=2-3]
\end{tikzcd}\]
This allows us to define a Global sections functor into the fiber of $\mathbb{B}$ over the point
\[\begin{tikzcd}[row sep=tiny]
	{\mathfrak{A}} & {\mathfrak{B}_*} \\
	{\mathbb{A}} & {{!_{\mid A \mid}}_*\iota(\mathbb{A})}
	\arrow[from=1-1, to=1-2, "\Gamma"]
	\arrow[from=2-1, to=2-2, shorten <=1pt, shorten >=1pt, maps to]
\end{tikzcd}\]
where $i_* :\mathfrak{B}_* \hookrightarrow \mathfrak{B}$ is the inclusion of the fiber at $*$. 
\end{remark}

\begin{lemma}\label{globalsections}
For a duality $ (\mathfrak{A}, \mathfrak{B}, \iota, \mathfrak{S})$, we have an adjunction $ \mathfrak{S} i_* \dashv \Gamma$
\end{lemma}

\begin{proof}
For one direction, let be $\mathbb{B}$ an object over the point, that is, $|\mathbb{B}|=*$, and $ f :\mathbb{B} \rightarrow \Gamma \mathbb{A}$. Then composing with the cartesian lift $\overline{{!_{\mid \mathbb{A} \mid}}} : \Gamma\mathbb{A} \rightarrow \iota_(\mathbb{A})$ induces from $ \mathfrak{S} \dashv \iota$ a unique map $   \mathfrak{S}B \rightarrow \mathbb{A}$ in $\mathfrak{A}$. Conversely, for any $   f: \mathfrak{S}B \rightarrow \mathbb{A}$, if we have $ \Gamma f: \Gamma\mathfrak{S}\mathbb{B} \rightarrow \Gamma\mathbb{A}$. But observe that the unit $ \eta_\mathbb{B} : \mathbb{B} \rightarrow \iota\mathfrak{S}\mathbb{B}$ induces a unique dashed arrow below from the opcartesiannes of $\overline{!_{|\iota\mathfrak{S}\mathbb{B}|}} $:
\[\begin{tikzcd}[row sep=small]
	{\Gamma\mathfrak{S}\mathbb{B}} \\
	{*} & {\mathbb{B}} & {\iota\mathfrak{S}\mathbb{B}} \\
	& {*} & {|\iota\mathfrak{S}\mathbb{B}|}
	\arrow["{\eta_\mathbb{B}}"', from=2-2, to=2-3]
	\arrow["{!_{|\iota\mathfrak{S}\mathbb{B}|}}", from=3-3, to=3-2]
	\arrow[Rightarrow, from=2-1, to=3-2, no head]
	\arrow["{!_{|\iota\mathfrak{S}\mathbb{B}|}}"'{very near start}, from=3-3, to=2-1, curve={height=6pt}]
	\arrow[from=2-2, to=1-1, dashed]
	\arrow[from=1-1, to=2-3, curve={height=-6pt}]
\end{tikzcd}\] and this is the unit of the restricted adjunction $ \mathfrak{S}i_* \dashv \Gamma$. Indeed, we can then compose $\Gamma f: \Gamma\mathfrak{S}\mathbb{B} \rightarrow \Gamma\mathbb{A}$ with $ B \rightarrow \Gamma \mathfrak{S}$ to get a map $ f :\mathbb{B} \rightarrow \Gamma \mathbb{A}$.
\end{proof}

\begin{remark}
This lemma shows that the fact we could restrict the spectral adjunction as developped in section 2 to the original version of Diers as recalled in section 1 is in fact inherent to the situation of duality and processes essentially from fibrational and opfibrational aspects, which are enacted through inverses and direct image in the concrete situations. 
\end{remark}

\begin{theorem}\label{Spectral dualities}
The construction above defines a 2-functor 
\[ \mathfrak{Spec}:  \mathfrak{Diers} \longrightarrow \mathfrak{Dual} \]
assigning to each Diers context $ (U,\mathcal{A}, \mathcal{B})$ the adjunction $ \Spec_U \dashv \iota_U$. 
\end{theorem}

\begin{proof}

The action of this functor on 0-cells was described in the second part. The remaining part of this section is aimed at making explicit how a morphism of Diers context could induce a 1-cell in the 2-category of geometric dualities.\\

In the following we fix a morphism of Diers contexts $(F,G) :(U_1, \mathcal{A}_1, \mathcal{B}_1) \rightarrow (U_2, \mathcal{A}_2, \mathcal{B}_2)   $. We construct a pair of functors $ \mathbb{F} : U_1-Spaces \rightarrow U_2-Spaces $ and $\mathbb{G} : \mathcal{B}_1-Spaces \rightarrow  \mathcal{B}_2-Spaces $ as follows. For each $\mathcal{B}_1$ space $((X,\tau), \mathbb{B})$, the sheaf $\mathbb{B}$ of $\mathcal{B}_1$ objects can be turned into a sheaf of $ \mathcal{B}_2$ objects thanks to $G$, not only without changing the base space, but also in a quite smooth manner: we have from composition with $G$ a presheaf $ G\mathbb{B} : \tau^{op} \rightarrow \mathcal{B}_2 $ which acts as $ u \mapsto G(\mathbb{B}(u))$; but in fact this presheaf is allready a sheaf, for $G$ is a right adjoint as a morphism of locally finitely presentable categories, so that for each $ u$ and each $(u_i)_{i \in I}$ with $ u =\bigcup_{i \in I} u_i$ the limit in the descent property for $\mathbb{B}$ is preserved by $G$, that is,
\[  G\mathbb{B}(u) = \lim \Big{(} \prod_{i \in I} G\mathbb{B}(u_i) \rightrightarrows \prod_{i,j \in I} G\mathbb{B}(u_i \cap u_j) \Big{ )} \]
Hence this defines a sheaf of $\mathcal{B}_2$ objects $ \mathbb{G}(\mathbb{B})$ on $(X,\tau)$, and for a morphism of sheaf, we define $\mathbb{G}(\phi) $ as the whiskering
\[\begin{tikzcd}
	{\tau^{op}} && {\mathcal{B}_1} & {\mathcal{B}_2}
	\arrow["{\mathbb{B}_1}"{name=0}, from=1-1, to=1-3, curve={height=-12pt}]
	\arrow["{\mathbb{B}_2}"{name=1, swap}, from=1-1, to=1-3, curve={height=12pt}]
	\arrow["{G}", from=1-3, to=1-4]
	\arrow[Rightarrow, "{\phi}", from=0, to=1, shorten <=2pt, shorten >=2pt]
\end{tikzcd}\]
This defines a functor $ \mathbb{G}$ sending a $\mathcal{B}$-space $((X,\tau), \mathbb{B})$ on $((X,\tau), \mathbb{G}(\mathbb{B}))$
This functor $ \mathbb{G}$ conveniently restricts along $\iota_{U_1}$ thanks to its relation with $F$. Let be  $((X,\tau), \mathbb{A}, (A_x)_{x \in X})$ a $U_1$-space; then $((X,\tau), \mathbb{G}(\mathbb{A}))$ canonically defines a $ U_2$-space: indeed, at each stalk $x$ we have $ \mathbb{A}\!\mid_x = U_1(A_x)$; but now, as $G$ is finitary, we have 
\begin{align*}
    \mathbb{G}(\mathbb{A})\!\mid_x &= \underset{x \in u}{\colim} \; G(\mathbb{A}(u)) \\
    &\simeq G(\underset{x \in u}{\colim} \; \mathbb{A}(u)) \\
    &\simeq G(\mathbb{A}\mid_x)\\
    &= GU_1(A_x) \\
    &\simeq U_2F(A_x) 
\end{align*}
where the last isomorphism comes from the natural isomorphism $ \theta : GU_1 \simeq U_2F $: hence $ \mathbb{G}(\mathbb{A})$ has $F(A_x)_{x \in X}$ as the local data attached to its stalks, and we can do the same for arrows. Hence we just have to put $\mathbb{F}(((X,\tau), \mathbb{A}, (A_x)_{x \in X}) = ((X,\tau), \mathbb{G}(\mathbb{A}), (F(A_x))_{x \in X}) $. This proves we have an invertible 2-cell
\[\begin{tikzcd}
	{U_1-Spaces} & {\mathcal{B}_1-Spaces} \\
	{U_2-Spaces} & {\mathcal{B}_2-Spaces}
	\arrow["{\mathbb{F}}"', from=1-1, to=2-1]
	\arrow["{\iota_{U_1}}", from=1-1, to=1-2]
	\arrow["{\mathbb{G}}", from=1-2, to=2-2]
	\arrow["{\iota_{U_2}}"', from=2-1, to=2-2]
	\arrow["{\theta \atop \simeq}" description, from=1-1, to=2-2, phantom, no head]
\end{tikzcd}\]
2-functoriality can be left as an exercice, for it is more tiresome than enlightening and does not seem to correspond to any known concrete situation. 
\end{proof}

\begin{remark}
In the following, we need to distinguish data relative to $U_1$ and data relative to $U_2$. To this end we put the corresponding index into exponent to precise if a spectral data is relative to $U_1$ or $U_2$: for instance the set $ D^1_B$ is the set of $U_1$-diagonally universal morphisms under $B$ in $\mathcal{B}_1$; in particular we denote as $ \mathcal{D}^1_\mathbb{B}$ the category of basic open in the $U_1$-spectral topology for a sheaf of $\mathcal{B}_1$-objects $\mathbb{B}$ as defined in section 2; similarly $D^1 : \mathcal{D}^1_\mathbb{B} \rightarrow \tau^{op}_{\mathbb{B}}$ is the corresponding basis of the $U_1$ spectral topology, and so on... 
\end{remark}

Now we turn to the description of the mate $\sigma$, which entangles the action of the associated spectra $ \Spec_{U_1}$ and $\Spec_{U_2}$. We want to exhibit a canonical 2-cell
\[\begin{tikzcd}
	{U_1-Spaces} & {\mathcal{B}_1-Spaces} \\
	{U_2-Spaces} & {\mathcal{B}_2-Spaces}
	\arrow["{\mathbb{F}}"{name=0, swap}, from=1-1, to=2-1]
	\arrow["{\Spec_{U_1}}"', from=1-2, to=1-1]
	\arrow["{\mathbb{G}}", from=1-2, to=2-2]
	\arrow["{\Spec_{U_2}}"{name=1}, from=2-2, to=2-1]
	\arrow[Rightarrow, "{\sigma}"', from=1, to=0, curve={height=12pt}, shorten <=5pt, shorten >=5pt]
\end{tikzcd}\]
In the following we fix a $\mathcal{B} $-space $((X,\tau), \mathbb{B})$. First, let us look at the the underlying map. Observe that $ G$ may not preserve diagonally universal morphisms: hence in each $ (x, \xi) \in X_\mathbb{B}$, we have a factorization 
\[\begin{tikzcd}
	{G(\mathbb{B}\!\mid_x)} && {GU_1(A_\xi) \simeq U_2F(A_\xi)} \\
	& {U_2(A_\xi)}
	\arrow["{G(\xi)}", from=1-1, to=1-3]
	\arrow["{\eta^{A_\xi}_{G(\xi)}}"', from=1-1, to=2-2]
	\arrow["{U_2L_{A_\xi}(G(\xi))}"', from=2-2, to=1-3]
\end{tikzcd}\]
Moreover we also have the following general property concerning the way $G$ interacts with stable factorization in $\mathcal{B}_1$:
 \begin{lemma}
For any $f : B \rightarrow U_1(A)$ we have that $ n_{G(\eta^{A}_f)}= \eta^{F(A)}_{G(f)}$.
\end{lemma}

\begin{proof}
Beware that there is no reason for $G$ to preserve diagonally universal morphisms. However we have the following factorization
\[
\begin{tikzcd}[row sep = large, column sep=large] 
G(B) \arrow[]{r}{G(f)} \arrow[]{d}[swap]{n_{G(f)}} \arrow[]{dr}[near start, swap]{G(n_f)} & U_2F(A) \\
U_2(L^2_{F(A)}(G(f))) \arrow[]{r}[swap]{U_2(\sigma^A_f)} \arrow[crossing over]{ru}[near end]{U_2(u_{G(f)})} &U_2F(L^1_A(f)) \arrow[]{u}[swap]{GU_1(u_{f})= U_2F(u_f)} &
\end{tikzcd}
\]
And postcomposing with $GU_1(u_{f})= U_2F(u_f) $ does not modify the factorization while $ U(\sigma^A_f)$ is a local morphism, so that $ n_{G(f)}$ is also the candidate for $ G(n_f)$. 
\end{proof}
 
Also make the following observation about the left adjoint of $G$:

\begin{lemma}\label{G^* is cool}
$G^*$ sends $U_2$-diagonally universal morphisms of finite presentation to $U_1$-diagonally universal morphisms of finite presentation.
\end{lemma}

\begin{proof}
Let be $n : B \rightarrow C$ in $ \mathcal{D}^2_B$ for $B$ in $\mathcal{B}_2$, and $ u : A \rightarrow A' $ in $\mathcal{A}_1$. Then for a square
\[\begin{tikzcd}
	{G^*(B)} & {U_1(A)} \\
	{G^*(C)} & {U_1(A')}
	\arrow["{G^*(n)}"', from=1-1, to=2-1]
	\arrow["{g}"', from=2-1, to=2-2]
	\arrow["{f}", from=1-1, to=1-2]
	\arrow["{U_1(u)}", from=1-2, to=2-2]
\end{tikzcd}\]
then by adjunction this square corresponds to a unique square in $\mathcal{B}_2$ which admist a filler as below 
\[\begin{tikzcd}
	{B} & {GU_1(A) \simeq U_2F(A)} \\
	{C} & {GU_1(A') \simeq U_2F(A')}
	\arrow["{n}"', from=1-1, to=2-1]
	\arrow["{\overline{g}}"', from=2-1, to=2-2]
	\arrow["{\overline{f}}", from=1-1, to=1-2]
	\arrow["{U_2F(A)}", from=1-2, to=2-2]
	\arrow["{d}" description, from=2-1, to=1-2, dashed]
\end{tikzcd}\]
and the map $C \rightarrow GU_1(A)$ itself corresponds again uniquely to a map $ \overline{d} : G^*C \rightarrow U_1(A)$ wich is a filler in $\mathcal{B}_1$.
\end{proof} 
 
\begin{proposition}
For any $((X,\tau), \mathbb{B})$ following map is continuous 
\[\begin{tikzcd}[column sep=large, row sep=tiny]
	{Spec_{U_1}(\mathbb{B}) } & {Spec_{U_2}(\mathbb{G}(\mathbb{B}))} \\
	{(x,\xi)} & {(x, \eta^{F(A_\xi)}_{G(\xi)})}
	\arrow["{\mathfrak{s}_{((X,\tau), \mathbb{B})}}", from=1-1, to=1-2]
	\arrow[from=2-1, to=2-2, shorten <=6pt, shorten >=6pt, maps to]
\end{tikzcd}\]
\end{proposition}

\begin{proof}
Let be $ (u,n) \in \mathcal{D}^2_{\mathbb{G}\mathbb{B}} $ and $(x, \xi)$ in $ \mathfrak{s}^{-1}_{((X,\tau), \mathbb{B})}(D^2_{(u,n)}) $: that is, suppose we have a factorization
\[\begin{tikzcd}
	{G(\mathbb{B}(u))} & {G(\mathbb{B}\!\mid_x)} \\
	{C} & {U_2(A_{G(\xi)})}
	\arrow["{n}"', from=1-1, to=2-1]
	\arrow["{s}"', from=2-1, to=2-2]
	\arrow["{\eta^{F(A_\xi)}_{G(\xi)}}", from=1-2, to=2-2]
	\arrow["{\rho^u_x}", from=1-1, to=1-2]
\end{tikzcd}\]
Then, for the left adjoint $ G^*$ preserves finitely presentedness and diagonally universal morphisms by the previous lemma, we know that $ G^*(n) : G^*G(\mathbb{B}(u)) \rightarrow G^*(C)$ is $U_1$-diagonally universal of finite presentation in $\mathcal{B}_1$. But then one can push along the $ G^*\dashv G$-counit, we have a factorization of $ \xi$ through this pushout
\[\begin{tikzcd}[column sep=large]
	{G^*G(\mathbb{B}(u))} & {\mathbb{B}(u)} \\
	{G^*(C)} & {{\epsilon_{\mathbb{B}(u)}}_*G^*(C)} \\
	& {G^*GU_1(A_\xi)} & {U_1(A_\xi)}
	\arrow["{\epsilon_{\mathbb{B}(u)}}", from=1-1, to=1-2]
	\arrow["{G^*(n)}"', from=1-1, to=2-1]
	\arrow["{G^*(s)}"', from=2-1, to=3-2]
	\arrow[from=2-1, to=2-2]	
	\arrow[from=1-1, to=3-2, curve={height=-12pt}, crossing over]
	\arrow["{{\epsilon_{\mathbb{B}(u)}}_*G^*(n)}" description, from=1-2, to=2-2]
	\arrow["\lrcorner"{very near start, rotate=180}, from=2-2, to=1-1, phantom]
	\arrow[from=2-2, to=3-3, dashed]
	\arrow[from=3-2, to=3-3, "\epsilon_{U_1(A_\xi)}"']
	\arrow["{\xi}", from=1-2, to=3-3, curve={height=-12pt}]
\end{tikzcd}\] 
This defines an open $ D^1_{(u, {\epsilon_{\mathbb{B}(u)}}_*G^*(n))}$ containing $(x,\xi)$.
But now any $ \xi'$ in $ D^1_{(u, {\epsilon_{\mathbb{B}(u)}}_*G^*(n))}$ must also be sent inside $ D^2_{(u,n)}$: indeed, if we have a factorization as below 
\[\begin{tikzcd}
	{G^*G(\mathbb{B}(u))} & {\mathbb{B}(u)} \\
	{G^*(C)} & {{\epsilon_{\mathbb{B}(u)}}_*G^*(C)} \\
	&& {U_1(A_\xi)}
	\arrow["{\epsilon_{\mathbb{B}(u)}}", from=1-1, to=1-2]
	\arrow["{G^*(n)}"', from=1-1, to=2-1]
	\arrow[from=2-1, to=2-2]
	\arrow["{{\epsilon_{\mathbb{B}(u)}}_*G^*(n)}" description, from=1-2, to=2-2]
	\arrow["\lrcorner"{very near start, rotate=180}, from=2-2, to=1-1, phantom]
	\arrow[from=2-2, to=3-3, dashed]
	\arrow["{\xi}", from=1-2, to=3-3, curve={height=-12pt}]
\end{tikzcd}\]
then applying again $ G$ and composing with unit ensures that $ n \leq G(\xi)$ as visualized below
\[\begin{tikzcd}
	{C} & {G(\mathbb{B}(u))} \\
	{GG^*(C)} & {G({\epsilon_{\mathbb{B}(u)}}_*G^*(C))} \\
	&& {GU_1(A_\xi)\simeq u_2F(A_\xi)}
	\arrow[from=2-1, to=2-2]
	\arrow["{G({\epsilon_{\mathbb{B}(u)}}_*G^*(n))}" description, from=1-2, to=2-2]
	\arrow[from=2-2, to=3-3, dashed]
	\arrow["{G(\xi)}", from=1-2, to=3-3, curve={height=-12pt}]
	\arrow["{n}"', from=1-2, to=1-1]
	\arrow["{\eta_C}"', from=1-1, to=2-1]
\end{tikzcd}\]
But then in particular one must have $ n \leq \eta^{F(A_\xi)}_{G(\xi)}$ by Diers condition. This proves that 
\[ \mathfrak{s}^{-1}_{((X,\tau), \mathbb{B})}(D^2_{(u,n)})  = D^1_{(u, {\epsilon_{\mathbb{B}(u)}}_*G^*(n))} \]
ensuring the continuity of $\mathfrak{s}_{((X,\tau), \mathbb{B})}$. 
\end{proof}

Now we turn to the construction of a morphism of sheaves $ \sigma^\sharp : \widetilde{\mathbb{G}\mathbb{B}} \rightarrow {\mathfrak{s}_{((X,\tau), \mathbb{B})}}_*\mathbb{G}\widetilde{\mathbb{B}}$. First, consider the expression of the structural presheaf $ \overline{\mathbb{B}}$ at $ D^2_{(u,n)}$ and apply $G$:
\begin{equation*}\begin{split}
    G \overline{\mathbb{B}}\mathfrak{s}^{-1}_{((X,\tau), \mathbb{B})}(D^2_{(u,n)}) &\simeq G \Big{(}\underset{\mathfrak{s}^{-1}_{((X,\tau), \mathbb{B})}(D^2_{(u,n)}) \subseteq D^1(v,m)}{\colim} \,\cod(m) \Big{)}\\
    &\simeq \underset{\mathfrak{s}^{-1}_{((X,\tau), \mathbb{B})}(D^2_{(u,n)}) \subseteq D^1(v,m)}{\colim} \,G(\cod(m)) 
 \end{split}\end{equation*} 
where the last iso comes form $G$ is finitary. The indexing set enumerates the pairs such that 
\[D^1_{(u, {\epsilon_{\mathbb{B}(u)}}_*G^*(n))} \subseteq D^1(v,m)  \]
But we always have $ n \leq G({\epsilon_{\mathbb{B}(u)}}_*G^*(n)))$ thanks to the unit as seen in the following diagram
\[\begin{tikzcd}
	& {G(\mathbb{B}(u))} \\
	{C} && {G({\epsilon_{\mathbb{B}(u)}}_*G^*(C))} \\
	& {GG^*(C)}
	\arrow[from=3-2, to=2-3]
	\arrow["{G({\epsilon_{\mathbb{B}(u)}}_*G^*(n))}", from=1-2, to=2-3]
	\arrow["{n}"', from=1-2, to=2-1]
	\arrow["{\eta_C}"', from=2-1, to=3-2]
	\arrow[from=2-1, to=2-3, dashed]
\end{tikzcd}\]
Hence the colimit can restrict to the $D^1_{(v,m)} $ with $u=v$ ad $ m = {\epsilon_{\mathbb{B}(u)}}_*G^*(C)$, and hence the cocone of the colimit, composed with the map above, defines a natural transformation
\[ ( s_{(u,n)} : \cod \Rightarrow G \mathbb{B} \mathfrak{s}^{-1}_{((X,\tau), \mathbb{B})} D )_{(u,n) \in \mathcal{D}^2_{\mathbb{G}\mathbb{B}}}\]
Hence we can invoke the universal property of the left Kan extension as below
\[\begin{tikzcd}[column sep=large]
	{\mathcal{D}^2_{\mathbb{G}\mathbb{B}}} && {\mathcal{B}_2} \\
	{\tau^{op}_{\mathbb{G}\mathbb{B}}} \\
	{\tau^{op}_{\mathbb{B}}} && {\mathcal{B}_1}
	\arrow["{D^2}"', from=1-1, to=2-1]
	\arrow["{\mathfrak{s}^{-1}_{((X,\tau), \mathbb{B})}}"', from=2-1, to=3-1]
	\arrow["{\overline{\mathbb{B}}}"', from=3-1, to=3-3]
	\arrow["{\cod}"{name=0}, from=1-1, to=1-3]
	\arrow["{\overline{\mathbb{G}\mathbb{B}}}"{name=1, swap}, from=2-1, to=1-3]
	\arrow["{G}"{name=2, swap}, from=3-3, to=1-3]
	\arrow[Rightarrow, "{\zeta}", from=0, to=1, shift right=1, shorten <=2pt, shorten >=5pt]
	\arrow[Rightarrow, "{s}"', from=0, to=2, shift right=10, curve={height=12pt}, shorten <=-1pt, shorten >=-1pt, crossing over]
	\arrow[Rightarrow, "{\sigma}", from=1, to=2, shift left=5, curve={height=6pt}, shorten <=5pt, shorten >=5pt, dashed]
\end{tikzcd}\]
Then the desired direct image part is obtained after factorizing through the sheafification. Now, to retrieve the inverse image part, recall that inverse images also are defined through left Kan extensions, and in our case, we have $\mathfrak{s}^{*}_{((X,\tau), \mathbb{B})} \overline{\mathbb{G}(\mathbb{B})} = \lan_{ \mathfrak{s}^{-1}_{((X,\tau), \mathbb{B})}}  \overline{\mathbb{G}(\mathbb{B})} $, so we have a natural map in the diagram below
\[\begin{tikzcd}[sep=huge]
	{\tau^{op}_{\mathbb{G}\mathbb{B}}} && {\mathcal{B}_2} \\
	{\tau^{op}_{\mathbb{B}}} && {\mathcal{B}_1}
	\arrow["{\mathfrak{s}^{-1}_{((X,\tau), \mathbb{B})}}"', from=1-1, to=2-1]
	\arrow["{\overline{\mathbb{B}}}"', from=2-1, to=2-3]
	\arrow["{\overline{\mathbb{G}\mathbb{B}}}"{name=0}, from=1-1, to=1-3]
	\arrow["{G}"{name=1, swap}, from=2-3, to=1-3]
	\arrow["{\lan_{ \mathfrak{s}^{-1}_{((X,\tau), \mathbb{B})}}  \overline{\mathbb{G}(\mathbb{B})}}"{name=2, description}, from=2-1, to=1-3]
	\arrow[Rightarrow, "{\sigma}", from=0, to=1, curve={height=20pt}, shorten <=22pt, shorten >=12pt, crossing over, shift left=8]
	\arrow[Rightarrow, from=0, to=2, shift right=2, shorten <=2pt, shorten >=2pt]
	\arrow[Rightarrow, from=2, to=2-3, shorten <=4pt, shorten >=4pt, dashed]
\end{tikzcd}\]
which induces after sheafification a morphism $\sigma^\flat : {\mathfrak{s}_{((X,\tau), \mathbb{B})}}^*\widetilde{\mathbb{G}\mathbb{B}} \rightarrow \mathbb{G}\widetilde{\mathbb{B}}$.

\begin{remark}
At the level of stalks, the inverse image comorphism happens to coincide with the mate. Indeed, as stalks are unchanged by sheafification, for a point $ (x,\xi) $ of $Spec_{U_1}(\mathbb{B})$ corresponding to a candidate $ \xi : \mathbb{B}\!\mid_x \rightarrow U_1(A_\xi)$:
\[ {\mathfrak{s}_{((X,\tau), \mathbb{B})}}^*\overline{\mathbb{G}\mathbb{B}}\!\mid_{(x,\xi)} = \overline{\mathbb{G}\mathbb{B}}\!\mid_{\eta^{F(A_\xi)}_{G(\xi)}} =  U_2(L^2_{F(A_x)}(G(\xi)))
\]
while we have 
\[ G(\overline{\mathbb{B}}_{(x,\xi)})=GU_1(A_\xi) = U_2F(A_\xi) \]
And the value of the comorphism at the stalk is actually the right part of the factorization
\[ \begin{tikzcd}[row sep=small]
G(\mathbb{B}\mid_x) \arrow[]{rr}{G(\xi)} \arrow[]{rd}[swap]{\phi(\xi) = n_{G(\xi)}} & & GU_1(A_x) = U_2F(A_x) \\ & U_2L^2_{A_x}({G(\xi))} \arrow[]{ru}[swap]{U_2(u_{G(n\xi)})}
\end{tikzcd} \] that is $ \sigma^\sharp_{(x, \xi)} = U_2(u_{G(\xi)}) $, which comes uniquely from the mate \[ u_{G(\xi)} = \sigma^{A_x}_{\xi} : L^2_{A_x}({G(\xi))} \rightarrow F(A_x) \]In particular this ensures that $ ({\mathfrak{s}_{((X,\tau), \mathbb{B})}}, \sigma)  $ is a morphism of $U_2$-locally structured spaces as it behaves correctly at stalks.
\end{remark}

We should conclude with the following partial inverse result, which allows one to reconstruct right multi-adjoints from dualities by taking the fibers over the point:

\begin{theorem}\label{MRAJ from duality}
Let be $ (\mathfrak{A}, \mathfrak{B}, \iota, \mathfrak{S})$ a duality. Then the functor \[ \mathfrak{A}_* \stackrel{{\iota}_*}{\longrightarrow} \mathfrak{B}_* \]
is a right multi-adjoint.
\end{theorem}

\begin{proof}
This exploits the properties of fibrations. Let be $ \mathbb{B} $ an object of the fiber $ \mathfrak{B}_*$: we prove that the comma category $ \mathbb{B}\downarrow \iota_*$ has a small initial family. For any $\mathbb{A}$ in $\mathfrak{A}$, a morphism $ f : \mathbb{B} \rightarrow \iota\mathbb{A}$ factorizes uniquely through $\mathfrak{S}\mathbb{B}$ as 
\[\begin{tikzcd}[ sep=large]
	& {\iota\mathfrak{S}\mathbb{B}} & {\mathbb{B}} \\
	{\iota\mathbb{A}} & {|\iota\mathfrak{S}\mathbb{B}|} & {*} \\
	{*}
	\arrow[Rightarrow, from=3-1, to=2-3, curve={height=12pt}, no head]
	\arrow["{|\iota g|}", from=3-1, to=2-2]
	\arrow["{!}", from=2-2, to=2-3]
	\arrow["{f}"', from=1-3, to=2-1, curve={height=-12pt}]
	\arrow["{\eta_\mathbb{B}}"', from=1-3, to=1-2]
	\arrow["{\iota g}", from=1-2, to=2-1]
\end{tikzcd}\]
so that $f$ defines a point $ |\iota g| : * \rightarrow |\iota\mathfrak{S}\mathbb{B}|$. Now, using the cartesianness of the fibration, there is a canonical lifting of this points, factorizing $f$ through the dashed vertical map as below:
\[\begin{tikzcd}
	& {|\iota g|_*(\iota\mathfrak{S}\mathbb{B})} & {\iota\mathfrak{S}\mathbb{B}} \\
	{\iota\mathbb{A}} & {*} & {|\iota\mathfrak{S}\mathbb{B}|} \\
	{*}
	\arrow[Rightarrow, from=3-1, to=2-2, no head]
	\arrow["{|\iota g|}"', from=3-1, to=2-3, curve={height=12pt}]
	\arrow["{|\iota g|}", from=2-2, to=2-3]
	\arrow["{\overline{|\iota g|}}"', from=1-3, to=1-2]
	\arrow["{\iota g}"', from=1-3, to=2-1, curve={height=-12pt}]
	\arrow[from=1-2, to=2-1, dashed]
\end{tikzcd}\]
and moreover this lift is actually in $\mathfrak{A}$ as it is itself equiped with a fibration, so we have in $\mathfrak{A}$
\[\begin{tikzcd}
	& {|\iota g|_*(\iota\mathfrak{S}\mathbb{B})} & {\mathfrak{S}\mathbb{B}} \\
	{\mathbb{A}} & {*} & {|\iota\mathfrak{S}\mathbb{B}|} \\
	{*}
	\arrow[Rightarrow, from=3-1, to=2-2, no head]
	\arrow["{|\iota g|}"', from=3-1, to=2-3, curve={height=12pt}]
	\arrow["{|\iota g|}", from=2-2, to=2-3]
	\arrow["{\overline{|\iota g|}}"', from=1-3, to=1-2]
	\arrow["{\iota g}", from=1-3, to=2-1, curve={height=-12pt}]
	\arrow[from=1-2, to=2-1, dashed]
\end{tikzcd}\]
Hence the functor associating to any point its cartesian lifts 
\[ \begin{array}{rcl}
     pt\mid \frak{S}\mathbb{B}  \mid & \longrightarrow &  \frak{S}\mathbb{B} \downarrow \iota_*   \\
    p & \longmapsto &   \overline{p} : \frak{S}\mathbb{B} \rightarrow p_*\frak{S}\mathbb{B}
\end{array}    \]
indexes an initial family in the comma under $ \mathbb{B}$. This exhibits $\iota_*$ as a right multi-adjoint.
\end{proof}

\printbibliography

\end{document}